\documentclass{amsart}
\usepackage{amssymb,amsmath,mathrsfs,amsfonts}
\usepackage{amscd, amssymb, amsmath, amsthm, graphics}
\usepackage{amsmath,amsfonts,amsthm,amssymb}
\usepackage{fancyhdr}
\usepackage[all]{xy}
\usepackage[colorlinks=true]{hyperref}

\newtheorem{theorem}{Theorem}[section]
\newtheorem{lemma}[theorem]{Lemma}
\newtheorem{proposition}[theorem]{Proposition}
\newtheorem{corollary}[theorem]{Corollary}

\theoremstyle{definition}
\newtheorem{definition}[theorem]{Definition}

\newtheorem{question}[theorem]{Question}
\newtheorem{remark}[theorem]{Remark}

\begin{document}
\sloppy

\title[Amalgamations of hyperbolic $3$-manifold groups are not LERF]{Geometrically finite amalgamations of hyperbolic $3$-manifold groups are not LERF}

    %Information for first author
\author{Hongbin Sun}
\address{Department of Mathematics, University of California at Berkeley, Berkeley, CA 94720, USA}
\email{hongbins@math.berkeley.edu}

\dedicatory{Dedicated to Professor Boju Jiang on his 80th birthday}

%    General info

\subjclass[2010]{20E26, 57M05, 57M50, 22E40}
\thanks{The author is partially supported by NSF Grant No. DMS-1510383.}
\keywords{locally extended residually finite, graph of groups, hyperbolic $3$-manifolds, arithmetic hyperbolic manifolds}

\date{\today}
\begin{abstract}
We prove that, for any two finite volume hyperbolic $3$-manifolds, the amalgamation of their fundamental groups along any nontrivial geometrically finite subgroup is not LERF. This generalizes the author's previous work on nonLERFness of amalgamations of hyperbolic $3$-manifold groups along abelian subgroups. A consequence of this result is that closed arithmetic hyperbolic $4$-manifolds have nonLERF fundamental groups. Along with the author's previous work, we get that, for any arithmetic hyperbolic manifold with dimension at least $4$, with possible exceptions in $7$-dimensional manifolds defined by the octonion, its fundamental group is not LERF.
\end{abstract}

\maketitle
\vspace{-.5cm}
%\newpage
\section{Introduction}

For a group $G$ and a subgroup $H < G$, we say that $H$ is {\it separable} in $G$ if for any $g\in G\setminus H$, there exists a finite index subgroup $G'<G$ such that $H<G'$ and $g\notin G'$. $G$ is {\it LERF} (locally extended residually finite) or {\it subgroup separable} if all finitely generated subgroups of $G$ are separable.

The LERFness of groups is a property closely related with low dimensional topology, especially the virtual Haken conjecture on $3$-manifolds (which is settled in \cite{Ag3}). The topological significance of LERFness is shown by the following picture: suppose we have a $\pi_1$-injective immersed compact object in a space $S\looparrowright M$ (e.g. a $\pi_1$-injective immersed surface in a $3$-manifold), if $\pi_1(S)$ is a separable subgroup of $\pi_1(M)$ (which holds if $\pi_1(M)$ is LERF), then $S$ lifts to an embedded object in some finite cover of $M$.

Among fundamental groups of low dimensional manifolds, it is known that following groups are LERF: free groups (\cite{Ha}), surface groups (\cite{Sc}) and geometric $3$-manifolds groups (the Seifert fibered space case, Sol manifold case and hyperbolic case are settled in \cite{Sc}, \cite{Ma} and \cite{Ag3,Wi} respectively). It is also know that fundamental groups of non-geometric $3$-manifolds are not LERF, including: groups of nontrivial graph manifolds (\cite{NW}), and groups of mixed $3$-manifolds (\cite{Sun}, and the first such example is given in \cite{Liu}). So the LERFness of fundamental groups of compact $1$-manifolds, compact $2$-manifolds and compact $3$-manifolds with empty or tori boundary is completely determined.

For any $n\geq 4$ and any finitely presented group $G$, there exists a closed smooth $n$-manifold with fundamental group isomorphic to $G$. So it is impossible to have a complete criterion of LERFness as dimension $\leq 3$ case, and we may restrict to some special class of manifolds. In \cite{Sun}, it is shown that, for all arithmetic hyperbolic manifolds with dimension at least $4$, with possible exceptions in closed $4$-dimensional manifolds and $7$-dimensional manifolds defined by the octonion, their fundamental groups are not LERF.

The result in \cite{Sun} on high dimensional arithmetic hyperbolic manifold groups is a corollary of $3$-dimensional results, including nonLERFness of mixed $3$-manifold groups, and another result on nonLERFness of $\mathbb{Z}$-amalgamations of finite volume hyperbolic $3$-manifold groups (\cite{Sun}). A special family of mixed $3$-manifold groups consists of $\mathbb{Z}^2$-amalgamations of hyperbolic $3$-manifold groups, so both results are about abelian amalgamations of finite volume hyperbolic $3$-manifold groups.

\bigskip

In this paper, we give a more general result on nonLERFness of amalgamations of finite volume hyperbolic $3$-manifold groups.

\begin{theorem}\label{main}
Let $M_1,M_2$ be two finite volume hyperbolic $3$-manifolds, $A$ be a nontrivial finitely generated group and $i_1:A\to \pi_1(M_1)$, $i_2:A\to \pi_1(M_2)$ be two injective homomorphisms with geometrically finite images, then the amalgamation $\pi_1(M_1)*_A \pi_1(M_2)$ is not LERF.
\end{theorem}

Theorem \ref{main} implies the most interesting cases of Theorem 1.3 and Theorem 1.4 of \cite{Sun}: $\mathbb{Z}^2$- and $\mathbb{Z}$-amalgamations of finite volume hyperbolic $3$-manifold groups are not LERF.

Theorem \ref{main} might be a little bit surprising. It is known that all finite volume hyperbolic $3$-manifolds have LERF fundamental groups (\cite{Ag3,Wi}), and geometrically finite subgroups are considered to be "nice" subgroups of hyperbolic $3$-manifold groups. However, when we take amalgamation of two such LERF groups along a nontrivial "nice" subgroup, we get a nonLERF group. The main reason is that finite volume hyperbolic $3$-manifold groups have a lot of virtually fibered surface subgroups (geometrically infinite subgroups). They are "not so nice" subgroups of hyperbolic $3$-manifold groups, from geometric group theory point of view. Theorem \ref{main} may suggest that, although finite volume hyperbolic $3$-manifold groups are LERF, they are kind of on the border of LERF groups and nonLERF groups.

In \cite{Sun}, we used nonLERFness of $\mathbb{Z}^2$- and $\mathbb{Z}$-amalgamations of finite volume hyperbolic $3$-manifold groups to prove that most arithmetic hyperbolic manifolds with dimension $\geq 4$ have nonLERF fundamental groups. The cases \cite{Sun} does not cover are closed arithmetic hyperbolic $4$-manifolds and arithmetic hyperbolic $7$-manifolds defined by the octonion. The author still can not solve the $7$-dimensional octonion case since we do not know whether these manifolds contain totally geodesic $3$-dimensional submanifolds, or $3$-manifold groups as their subgroups, so our tool in $3$-manifold topology is not applicable. The closed $4$-dimensional case could not be solved in \cite{Sun} since two $3$-manifolds in a $4$-manifold usually intersect along a surface, whose group is not abelian, while \cite{Sun} only dealt with abelian amalgamations. The case that two hyperbolic $3$-manifolds intersecting along a totally geodesic subsurface is covered by Theorem \ref{main}. So we can prove the following result, which is one of the motivation of Theorem \ref{main}.

\begin{theorem}\label{arithmetichyperbolic4}
For any closed arithmetic hyperbolic $4$-manifold, its fundamental group is not LERF.
\end{theorem}

Along with Theorem 1.1 and 1.2 of \cite{Sun}, we get the following corollary.
\begin{corollary}\label{arithmetichyperbolic}
For any arithmetic hyperbolic manifold with dimension at least $4$ and not defined by the octonion (which only show up in dimension $7$), its fundamental group is not LERF.
\end{corollary}

This corollary seems also suggest that hyperbolic $3$-manifold groups are on the border of LERF groups and nonLERF groups, since if we increase the dimension by $1$, we goes from LERF groups to nonLERF groups. This is mainly because that high dimensional arithmetic hyperbolic manifold groups contain many geometrically infinite subgroups which are not as nice as virtually fibered subgroups in dimension $3$. Note that for arithmetic hyperbolic manifolds of simplest type (defined by quadratic forms over totally real number fields), it is known that their geometrically finite subgroups are separable (\cite{BHW}).

\bigskip

The main idea of the proof of Theorem \ref{main} is similar to the proof of Theorem 1.3 and 1.4 of \cite{Sun}. To realize the idea as a mathematical proof, we need to take care of the following two points.

The first point is to show that $\pi_1(M_1)*_A \pi_1(M_2)$ has a subgroup with nontrivial induced graph of group structure that is "algebraically fibered". For an amalgamation $\pi_1(M_1)*_A \pi_1(M_2)$, by saying it is "algebraically fibered" , we mean that there are two fibered cohomology classes in $H^1(M_1;\mathbb{Z})$ and $H^1(M_2;\mathbb{Z})$ respectively, such that their restrictions on $H^1(A;\mathbb{Z})$ are the same nontrivial cohomology class. We care about such algebraically fibered structures since nonseparable subgroups we will get are constructed by "pasting" fibered surface subgroups in different vertex pieces together carefully. The "algebraically fibered" structure makes this pasting construction along $A$ applicable.

In \cite{Sun}, the existence of an "algebraically fibered" structure on a subgroup of $\pi_1(M_1)*_{\mathbb{Z}} \pi_1(M_2)$ is almost for free, and the existence of an "algebraically fibered" structure on a subgroup of $\pi_1(M_1)*_{\mathbb{Z}^2} \pi_1(M_2)$ follows from the work of Przytycki-Wise (\cite{PW}). However, for $\pi_1(M_1)*_A \pi_1(M_2)$ with a general group $A$, we need to prove the existence of an "algebraically fibered" structure. Moreover, for a general group $A$, the corresponding topological space is usually not a genuine fiber bundle over circle, so we need the notion of "algebraically fibered" structure here.

We will use Agol's construction of virtual fibered structures (\cite{Ag2}) and the virtual retract property of geometrically finite subgroups (\cite{CDW}) to prove the existence of an "algebraically fibered" structure on a subgroup of $\pi_1(M_1)*_A \pi(M_2)$ with nontrivial induced graph of group structure. The precise statement and its proof is given in Section \ref{algebraicfibering}.

The second point is that nonseparable subgroups we will construct are usually infinitely presented, so they may not be carried by $\pi_1$-injective compact objects. The nonseparable subgroups constructed in \cite{Sun} are always realized by $\pi_1$-injective compact objects (surfaces or one point unions of surfaces). However, for a general geometrically finite subgroup $A$ of a finite volume hyperbolic $3$-manifold group $\pi_1(M)$, and a cohomology class $\alpha\in H^1(M;\mathbb{Z})$, if we consider $\alpha$ as a homomorphism $\alpha:\pi_1(M)\to \mathbb{Z}$, the kernel of its restriction on $A$ might be infinitely generated.

Since we will "paste" fibered surface subgroups in different vertex pieces to get nonseparable subgroups, we may need to "paste" them along infinitely generated subgroups, and get finitely generated infinitely presented nonseparable subgroups. If we want to realize a finitely generated infinitely presented subgroup by a $\pi_1$-injective object, it must be noncompact. However, since Scott's topological interpretation of separability (\cite{Sc}) do need a compact object in the corresponding covering space, we will sacrifice the $\pi_1$-injectivity and make sure that the nonseparable subgroup is carried by a compact object.

The organization of this paper is as the following. In Section \ref{preliminary}, we will review some relevant background on geometric group theory, topology of hyperbolic $3$-manifolds, and arithmetic hyperbolic manifolds. In Section \ref{algebraicfibering}, we will show that any nontrivial geometrically finite amalgamation $\pi_1(M_1)*_A \pi_1(M_2)$ of two finite volume hyperbolic $3$-manifold groups has a subgroup that has nontrivial induced graph of group structure and is algebraically fibered. In Section \ref{nonseparablesubgroup}, we give the proof of Theorem \ref{main}. In the proof, we first construct a further subgroup of the algebraically fibered group we got in Section \ref{algebraicfibering}, and construct a topological space $X$ with $\pi_1(X)$ isomorphic to this subgroup. Then we construct a map $f:Z\to X$ from a compact $2$-dimensional complex $Z$ to $X$, and show that $f_*(\pi(Z))<\pi_1(X)$ is not separable. In Section \ref{arithmetic4section}, we give the proof of Theorem \ref{arithmetichyperbolic4}. In Section \ref{furtherquestions}, we ask some further questions related to results in this paper.

\subsection*{Acknowledgement}
The author thanks Ian Agol and Daniel Groves for valuable conversations.

\section{Preliminaries}\label{preliminary}

In this section, we review some relevant background on geometric group theory, topology of hyperbolic $3$-manifolds, and arithmetic hyperbolic manifolds. This section has some overlap with Section 2 of \cite{Sun}.

\subsection{Subgroup separability}

In this subsection, we review basic concepts and properties on subgroup separability.

\begin{definition}
Let $G$ be a group, and $H<G$ be a subgroup, we say that $H$ is {\it separable} in $G$ if for any $g\in G\setminus H$, there exists a finite index subgroup $G'<G$ such that $H<G'$ and $g\notin G'$.
\end{definition}

It is obvious that finite index subgroups are always separable, so we are mainly interested in infinite index subgroups when we talk about separability.

\begin{definition}
A group $G$ is {\it LERF} (locally extended residually finite) or {\it subgroup separable} if all finitely generated subgroups of $G$ are separable in $G$.
\end{definition}

Here are two basic results on LERFness of groups, and we will use them implicitly in this paper.
\begin{itemize}
  \item If $A$ and $B$ are two LERF groups, then $A*B$ is also LERF.
  \item If $G$ is a group and $G'<G$ is a finite index subgroup, then $G$ is LERF if and only if $G'$ is LERF.
\end{itemize}

A more elementary property on LERFness of groups is that any subgroup of a LERF group is still LERF, and we state it as in the following lemma. As in \cite{Sun}, this property is fundamental for our proof on nonLERFness of groups: to prove a group is not LERF, we only need to find a manageable subgroup (e.g. amalgamation of $3$-manifold groups) and show this subgroup is not LERF.

\begin{lemma}\label{subgroup}
Let $G$ be a group and $\Gamma<G$ be a subgroup. For a further subgroup $H<\Gamma$, if $H$ is separable in $G$, then $H$ is separable in $\Gamma$.

In particular, if $\Gamma$ is not LERF, then $G$ is not LERF.
\end{lemma}

\subsection{Fibered structures of $3$-manifolds and quasi-fibered classes}

In this subsection, we always assume $3$-manifolds are compact, connected, oriented, irreducible and with empty or tori boundary.

By a fibered structure of a $3$-manifold, we mean a surface bundle over circle structure. For a fibered structure of $M$, there exist a compact oriented surface $S$ and an orientation-preserving surface automorphism $f:S\to S$, such that $M$ is homeomorphic to the mapping torus $S\times I/(s,0)\sim (f(s),1)$ with respect to the fibered structure. So we get a homology class $[S,\partial S]\in H_2(M,\partial M;\mathbb{Z})$ and its dual cohomology class $\alpha\in H^1(M;\mathbb{Z})$. They both correspond to the fibered structure of $M$, and we also consider $\alpha \in H^1(M;\mathbb{Z})$ as a homomorphism $\alpha: \pi_1(M)\to \mathbb{Z}$.

If a $3$-manifold $M$ has a fibered structure and $b_1(M)>1$, then $M$ has infinitely many fibered structures. These fibered structures are organized by the Thurston norm on $H^1(M; \mathbb{R})$ (defined in \cite{Th2}). We will not give the definition of Thurston norm here, but we need the following facts on Thurston norm (see \cite{Th2}).

In general, the Thurston norm is only a semi-norm, and it is a genuine norm when $M$ is a finite volume hyperbolic $3$-manifold. The unit ball of Thurston norm is a polyhedron in $H^1(M;\mathbb{R})$ with finitely many faces. For a top dimensional open face $F$ of the Thurston norm unit ball, let $C$ be the open cone over $F$. In \cite{Th2}, Thurston proved that an integer point $\alpha\in H^1(M;\mathbb{Z})\subset H^1(M;\mathbb{R})$ corresponds to a fibered structure of $M$ if and only if $\alpha$ is contained in an open cone $C$ as above and all integer points in $C$ correspond to fibered structures of $M$. In this case, the open cone $C$ is called a {\it fibered cone}. For any real coefficient cohomology class $\alpha\in C\subset H^1(M;\mathbb{R})$, it is called a {\it fibered class}.

All integer fibered classes in $H^1(M;\mathbb{Z})$ correspond to genuine fibered structures of $M$, and their image in $PH^1(M;\mathbb{Q})$ (the projectivization of $H^1(M;\mathbb{Q})$) form an open subset of $PH^1(M;\mathbb{Q})$. In particular, we have the following lemma on fibered classes.

\begin{lemma}\label{perturbation}
Suppose that $\alpha\in H^1(M;\mathbb{Z})$ is a fibered class of a $3$-manifold $M$, then for any $\beta \in H^1(M;\mathbb{Z})$, there exists $N\in \mathbb{Z}^+$, such that for any $n>N$, $n\alpha\pm \beta\in H^1(M;\mathbb{Z})$ are both fibered classes.
\end{lemma}

The following definition of quasi-fibered class is given in \cite{FV}:
\begin{definition}
For a cohomology class $\alpha \in H^1(M;\mathbb{Z})-\{0\}$, $\alpha$ is a {\it quasi-fibered class} if $\alpha$ lies on the closure of a fibered cone $C$.
\end{definition}

In the proof of the virtual fibering conjecture, the last step is Agol's criteria for virtual fiberings (\cite{Ag2}). In \cite{Ag2} (see an alternative proof in \cite{FK}), Agol showed that if the fundamental group of a compact irreducible $3$-manifold $M$ with empty or tori boundary satisfies the RFRS property (residually finite rational solvable), then $M$ is virtually fibered. We will not give the definition of RFRS here. Along with Wise and his collaborators' works on geometric group theory (e.g. \cite{Wi}), Agol showed that all finite volume hyperbolic $3$-manifolds have virtually RFRS fundamental group (\cite{Ag3}), and solved the virtual fibering conjecture.

In \cite{Ag2}, Agol actually proved that, if $\pi_1(M)$ is RFRS, then for any nontrivial non-fibered cohomology class $\alpha\in H^1(M;\mathbb{Z})$, there exists a finite cover $f:M'\to M$ such that $f^*(\alpha)$ lies on the boundary of a fibered cone of $H^1(M';\mathbb{R})$, i.e. $f^*(\alpha)$ is a quasi-fibered class. In \cite{FV}, the following proposition on quasi-fibered classes is proved (Corollary 5.2 of \cite{FV}). Its proof is a direct application of the result in \cite{Ag2}, and we modify the statement in \cite{FV} a little bit.

\begin{theorem}\label{quasifiber}
Let $M$ be a $3$-manifold with virtually RFRS fundamental group, then there exists a finite regular cover $p:M'\to M$, such that for any nontrivial class $\alpha \in H^1(M;\mathbb{Z})$, $p^*(\alpha)\in H^1(M';\mathbb{Z})$ is a quasi-fibered class.
\end{theorem}

So the process of pulling back a cohomology class to get a quasi-fibered class not only work for each cohomology class individually, but also work for all of them simultaneously.

\subsection{Infinite volume hyperbolic $3$-manifolds}

For infinite volume hyperbolic $3$-manifolds with finitely generated fundamental groups, there is a rich theory on such manifolds. In this paper, we are mainly interested in such manifolds that cover finite volume hyperbolic $3$-manifolds.

For a finite volume hyperbolic $3$-manifold $M$, we have the following dichotomy for a finitely generated infinite index subgroup $A<\pi_1(M)$ (the proof of this dichotomy is a combination of results in \cite{Th1}, \cite{Ca} and \cite{Ag1,CG}):
\begin{enumerate}
\item $A$ is a geometrically finite subgroup of $\pi_1(M)$. Equivalently, $A$ is (relatively) quasiconvex in the (relative) hyperbolic group $\pi_1(M)$, from geometric group theory point of view.
\item $A$ is a geometrically infinite subgroup of $\pi_1(M)$. It is equivalent to that $A$ is a virtually fibered surface subgroup of $M$.
\end{enumerate}
Here we do not give the definition of geometrically finite and geometrically infinite subgroups, the readers only need to know that if $A$ is not a virtually fibered surface subgroup, then it is a geometrically finite subgroup.

If $A<\pi_1(M)$ is a nontrivial finitely generated infinite index subgroup, the following lemma implies $b_1(A)\geq 1$. Actually, it holds for any nontrivial infinitely covolume discrete torsion-free subgroup of $\text{Isom}_+\mathbb{H}^3$, and it is well-known for experts on hyperbolic $3$-manifolds.

\begin{lemma}\label{b1positive}
If $A$ is a nontrivial finitely generated subgroup of $\text{Isom}_+\mathbb{H}^3$ that acts freely and properly discontinuously on $\mathbb{H}^3$ with infinite covolume, then the first betti number $b_1(A)\geq 1$.
\end{lemma}

\begin{proof}
Since $A$ acts freely and properly discontinuously on $\mathbb{H}^3$, $Y=\mathbb{H}^3/A$ is an infinitely volume hyperbolic $3$-manifold with finitely generated fundamental group.

By the tameness of hyperbolic $3$-manifolds (\cite{Ag1},\cite{CG}), $Y$ is homeomorphic to the interior of a compact $3$-manifold $X$.

Since $Y=\mathbb{H}^3/A$ has infinite volume, $X$ can not be a closed $3$-manifold. Since $Y$ is a quotient of $\mathbb{H}^3$, it is irreducible, and so does $X$. So no component of $\partial X$ is a sphere. Otherwise $X$ can only be a $3$-ball, which contradicts with that $A$ is a nontrivial group.

So $X$ is a compact $3$-manifold with boundary and each component of $\partial X$ has positive genus. Then a canonical application of the duality theorem on $3$-manifolds (half lives, half dies) implies $$b_1(A)=b_1(X)\geq \frac{1}{2}b_1(\partial X)\geq 1.$$
\end{proof}

Although Lemma \ref{b1positive} is well-known, it is fundamental for this paper. It is a classical result that amalgamations along trivial subgroups (free products) of LERF groups are still LERF, but Theorem \ref{main} implies nontrivial geometrically finite amalgamations of finite volume hyperbolic $3$-manifold groups (which are LERF groups) are not LERF. The difference is mainly rooted in Lemma \ref{b1positive}, since we will use $b_1(A)\geq 1$ seriously in the proof of Theorem \ref{main}.

With this result, for a (fibered) cohomology class $\alpha\in H^1(M;\mathbb{Z})$, its restriction on $A$ might be a nontrivial homomorphism to $\mathbb{Z}$.

\subsection{Virtual retractions of hyperbolic $3$-manifold groups}

\begin{definition}
For a group $G$ and a subgroup $A<G$, we say that $A$ is a {\it virtual retraction} of $G$ if there exists a finite index subgroup $G'<G$ and a homomorphism $\phi:G'\to A$, such that $A<G'$ and $\phi|_A=id_A$.
\end{definition}

In \cite{CDW}, it is shown that (relatively) quasiconvex subgroups of virtually compact special (relative) hyperbolic groups are virtual retractions. The celebrated virtually compact special theorem of Wise and Agol (\cite{Wi} for cusped manifolds and \cite{Ag3} for closed manifolds) implies that groups of finite volume hyperbolic $3$-manifolds are virtually compact special. These two results together imply the following theorem.

\begin{theorem}\label{virtualretract}
Let $M$ be a finite volume hyperbolic $3$-manifold, $A<\pi_1(M)$ be a geometrically finite subgroup (i.e. $A$ is not a virtually fibered surface subgroup), then $A$ is a virtual retraction of $\pi_1(M)$.
\end{theorem}

\subsection{Arithmetic hyperbolic manifolds}

In this subsection, we briefly review the definition of arithmetic hyperbolic manifolds of simplest type. They are defined by quadratic forms over number fields. Since all arithmetic hyperbolic $4$-manifolds are in the simplest type, this definition is sufficient for this paper. Most material in this subsection can be found in Chapter 6 of \cite{VS}.

Recall that the hyperboloid model of $\mathbb{H}^n$ is defined as the following. Equip $\mathbb{R}^{n+1}$ with a bilinear form $B:\mathbb{R}^{n+1}\times \mathbb{R}^{n+1}\to \mathbb{R}$ defined by $$B\big((x_1,\cdots,x_n,x_{n+1}),(y_1,\cdots,y_n,y_{n+1})\big)=x_1y_1+\cdots+x_ny_n-x_{n+1}y_{n+1}.$$ Then the hyperbolic space $\mathbb{H}^n$ can be identified with $$I^n=\{\vec{x}=(x_1,\cdots,x_n,x_{n+1})\ |\ B(\vec{x},\vec{x})=-1,x_{n+1}>0\}.$$ The hyperbolic metric is induced by the restriction of $B(\cdot,\cdot)$ on the tangent space of $I^n$.

The isometry group of $\mathbb{H}^n$ consists of all linear transformations of $\mathbb{R}^{n+1}$ that preserve $B(\cdot,\cdot)$ and fix $I^n$. Let $J=\text{diag}(1,\cdots,1,-1)$ be the $(n+1)\times (n+1)$ matrix defining the bilinear form $B(\cdot,\cdot)$, then the isometry group of $\mathbb{H}^n$ is $$\text{Isom}(\mathbb{H})\cong PO(n,1;\mathbb{R})=\{X\in GL(n+1,\mathbb{R})\ |\ X^tJX=J\}/(X\sim -X).$$
The orientation preserving isometry group of $\mathbb{H}^n$ is
$$\text{Isom}_+(\mathbb{H}^n)\cong SO_0(n,1;\mathbb{R}),$$ which is the component of $$SO(n,1;\mathbb{R})=\{X\in SL(n+1,\mathbb{R})\ |\ X^tJX=J\}$$ that contains the identity matrix.

Let $K\subset \mathbb{R}$ be a totally real number field, and $\sigma_1=id,\sigma_2,\cdots,\sigma_k$ be all the embeddings of $K$ into $\mathbb{R}$. Let $$f(x)=\sum_{i,j=1}^{n+1}a_{ij}x_ix_j,\ a_{ij}=a_{ji}\in K$$ be a nondegenerate quadratic form defined over $K$ with negative inertia index $1$ (as a quadratic form over $\mathbb{R}$). If for any $l>1$, the quadratic form $$f^{\sigma_l}(x)=\sum_{i,j=1}^{n+1}\sigma_l(a_{ij})x_ix_j$$ is positive definite, then we can use $K$ and $f$ to define an arithmetic hyperbolic group.

Let $A$ be the $(n+1)\times (n+1)$ matrix defining the quadratic form $f$. Since the negative inertia index of $A$ is $1$, the {\it special orthogonal group of $f$}: $$SO(f;\mathbb{R})=\{X\in SL(n+1,\mathbb{R})\ |\ X^tAX=A\}$$ is conjugate to $SO(n,1;\mathbb{R})$ by a matrix $P$ (satisfying $P^tAP=J$). $SO(f;\mathbb{R})$ has two components, and let $SO_0(f;\mathbb{R})$ be the component that contains the identity matrix.

Let $\mathcal{O}_K$ be the ring of algebraic integers in the field $K$. Then we can form the set of algebraic integer points $$SO(f;\mathcal{O}_K)=\{X\in SL(n+1,\mathcal{O}_K)\ |\ X^tAX=A\}\subset SO(f;\mathbb{R}).$$
The theory of arithmetic groups implies that $$SO_0(f;\mathcal{O}_K)=SO(f;\mathcal{O}_K)\cap SO_0(f;\mathbb{R})$$ conjugates to a lattice of $\text{Isom}_+(\mathbb{H}^n)$ (by the matrix $P$), i.e. the corresponding quotient space of $\mathbb{H}^n$ has finite volume. For simplicity, we still use $SO_0(f;\mathcal{O}_K)$ to denote its $P$-conjugation in $SO_0(n,1;\mathbb{R})\cong\text{Isom}_+(\mathbb{H}^n)$.

Here $SO_0(f;\mathcal{O}_K)\subset \text{Isom}_+(\mathbb{H}^n)$ is called {\it the arithmetic group} defined by the number field $K$ and quadratic form $f$, and $\mathbb{H}^n/SO_0(f;\mathcal{O}_K)$ is a finite volume hyperbolic arithmetic $n$-orbifold. A hyperbolic $n$-manifold (orbifold) $M$ is called {\it an arithmetic hyperbolic $n$-manifold (orbifold) of simplest type} if $M$ is commensurable with $\mathbb{H}^n/SO_0(f;\mathcal{O}_K)$ for some $K$ and $f$.

For this paper, the most important property of arithmetic hyperbolic manifolds of simplest type is that they contain many finite volume hyperbolic $3$-manifolds as totally geodesic submanifolds. This can be seen by diagonalizing the matrix $A$ and taking indefinite $4\times 4$ submatrices.

\section{Algebraically fibered structures on $\pi_1(M_1)*_A\pi_1(M_2)$}\label{algebraicfibering}

In this section, we construct "algebraically fibered" structures on certain subgroups of geometrically finite amalgamations of finite volume hyperbolic $3$-manifold groups. The construction of an algebraically fibered structure on certain subgroup is the first step for building an ideal model of geometrically finite amalgamations for which we can construct nonseparable subgroups.

\begin{definition}
For a group $G$, by an {\it algebraically fibered} structure on $G$, we mean a nontrivial homomorphism $G\to \mathbb{Z}$ with finitely generated kernel.
\end{definition}

The main result in this section is the following theorem, and we will prove it in subsection \ref{proof}. In subsection \ref{notsorelated}, we will prove a related result on virtually fibered boundary slopes of cusped hyperbolic $3$-manifolds, which is quite interesting by itself.

\begin{theorem}\label{matchingfibering}
Let $M_1,M_2$ be two finite volume hyperbolic $3$-manifolds, $A$ be a nontrivial finitely generated group, $i_1:A\to \pi_1(M_1)$ and $i_2:A\to \pi_1(M_2)$ be two injective homomorphisms with geometrically finite images. Then there exist finite covers $N_1$, $N_2$ of $M_1$, $M_2$ respectively, such that the following holds.
\begin{enumerate}
  \item $i_1^{-1}(\pi_1(N_1))$ and $i_2^{-1}(\pi_1(N_2))$ are the same subgroup $A'\leq A$ (and we identify $A'$ with their images in $\pi_1(N_1)$ and $\pi_1(N_2)$).
  \item There exist fibered classes $\alpha_1\in H^1(N_1;\mathbb{Z})$ and $\alpha_2\in H^1(N_2;\mathbb{Z})$, such that $\alpha_1|_{A'}=\alpha_2|_{A'}$ as homomorphisms from $A'$ to $\mathbb{Z}$, and the restricted homomorphisms are surjective.
\end{enumerate}
Moreover, the group $\pi_1(M_1)*_{A}\pi_1(M_2)$ contains a subgroup isomorphic to $\pi_1(N_1)*_{A'}\pi_1(N_2)$. The subgroup $\pi_1(N_1)*_{A'}\pi_1(N_2)$ is still a nontrivial geometrically finite amalgamation of finite volume hyperbolic $3$-manifold groups, and it has an algebraically fibered structure.
\end{theorem}

\subsection{Construct algebraically fibered structures on subgroups}\label{proof}

In this subsection, we give the proof of Theorem \ref{matchingfibering}. Before proving Theorem \ref{matchingfibering}, we start with the following proposition.

\begin{proposition}\label{extendtoquasifiber}
Let $M$ be a finite volume hyperbolic $3$-manifold, and $A<\pi_1(M)$ be a nontrivial geometrically finite subgroup. Then for any nontrivial homomorphism $\gamma: A\to \mathbb{Z}$, there exist a finite cover $M'$ of $M$ and a quasi-fibered class $\beta\in H^1(M';\mathbb{Z})$, such that the following hold.

For the subgroup $A'=A\cap \pi_1(M')$, $\beta|_{A'}=\gamma|_{A'}$ holds as homomorphisms from $A'$ to $\mathbb{Z}$.
\end{proposition}

\begin{proof}
Since $A<\pi_1(M)$ is a geometrically finite subgroup, by Theorem \ref{virtualretract}, there exists a finite cover $M''$ of $M$, such that $A<\pi_1(M'')$ and there is a retract homomorphism $\phi: \pi_1(M'')\to A$.

By taking the composition, we get a nontrivial homomorphism $\delta=\gamma\circ \phi: \pi_1(M'')\to \mathbb{Z}$, which gives a cohomology class $\delta\in H^1(M'';\mathbb{Z})$.

Then by Theorem \ref{quasifiber}, there exists a finite cover $p:M'\to M''$, such that $\beta=p^*(\delta)\in H^1(M';\mathbb{Z})$ is a quasi-fibered class.

For $A'=A\cap \pi_1(M')$, it is easy to see that $$\beta|_{A'}=p^*(\delta)|_{A'}=\delta\circ p_*|_{A'}=\gamma\circ \phi \circ p_*|_{A'}=\gamma\circ \phi|_{A'}=\gamma|_{A'}.$$

\end{proof}

Now we are ready to prove Theorem \ref{matchingfibering}. Actually, what we will prove is stronger than the statement of Theorem \ref{matchingfibering}. We can start with any nontrivial homomorphism $\gamma:A\to \mathbb{Z}$, then the restriction of the algebraically fibered structure on $A'$ can be arbitrarily close to the restriction of $\gamma$ on $A'$ (as homomorphisms to $\mathbb{Z}$).

\begin{proof}

By abusing notation, we still use $A$ to denote its images in $\pi_1(M_1)$ and $\pi_1(M_2)$. Lemma \ref{b1positive} implies that there exists a nontrivial homomorphism $\gamma:A\to \mathbb{Z}$.

By Proposition \ref{extendtoquasifiber}, there exist a finite cover $M_1'\to M_1$ and a quasi-fibered class $\beta_1'\in H^1(M_1';\mathbb{Z})$, such that for $A''=A\cap \pi_1(M_1')$, $\beta_1'|_{A''}=\gamma|_{A''}$ holds.

Since $\beta_1'$ is a quasi-fibered class of $M_1'$, there exists a fibered class $\beta_1\in H^1(M_1';\mathbb{Z})$ which is (arbitrarily) close to $\beta_1'$ in $PH^1(M_1';\mathbb{Q})$ (under the projectivization). Moreover, we can assume that $\beta_1|_{A''}:A''\to\mathbb{Z}$ is a nontrivial homomorphism.

Since $\pi_1(M_2)$ is LERF, $A''<\pi_1(M_2)$ is separable. So there exists a finite cover $M_2'$ of $M_2$ such that $A\cap\pi_1(M_2')=A''$. Now we apply Proposition \ref{extendtoquasifiber} again to $A''<\pi_1(M_2')$ and $\beta_1|_{A''}:A''\to \mathbb{Z}$. Then there exist a finite cover $N_2'$ of $M_2'$ and a quasi-fibered class $\beta_2\in H^1(N_2';\mathbb{Z})$, such that for $A'=A''\cap \pi_1(N_2')$, $\beta_2|_{A'}=\beta_1|_{A'}$ holds. Moreover, since $A'<A''$ is a finite index subgroup, $\beta_1|_{A'}$ is a nontrivial homomorphism from $A'$ to $\mathbb{Z}$.

Since $\pi_1(M_1')$ is LERF, $A'<\pi_1(M_1')$ is separable. So there exists a finite cover $N_1'$ of $M_1'$ such that $A''\cap \pi_1(N_1')=A'$. Moreover, since $A'<\pi_1(N_1')$ is a geometrically finite subgroup, it is a virtual retraction of a finite index subgroup of $\pi_1(N_1')$. By abusing notation, we denote the corresponding finite cover (such that $A'$ is a retraction of $\pi_1(N_1')$) by $p_1:N_1'\to M_1'$.

As a summary, we are in the following situation now. We have finite covers $N_1'$ and $N_2'$ of $M_1$ and $M_2$ respectively, with $A'=A\cap \pi_1(N_1')=A\cap \pi_1(N_2')$ and $A'$ is a retraction of $\pi_1(N_1')$ via a retract homomorphism $\phi:\pi_1(N_1')\to A'$. Moreover, we have a fibered class $p_1^*(\beta_1)\in H^1(N_1';\mathbb{Z})$ and a quasi-fibered class $\beta_2\in H^1(N_2';\mathbb{Z})$, such that $p_1^*(\beta_1)|_{A'}=\beta_2|_{A'}$ are nontrivial homomorphisms to $\mathbb{Z}$. Now we need to modify $\beta_2$ to a fibered class.

Since $\beta_2\in H^1(N_2';\mathbb{Z})$ is a quasi-fibered class, there exists $\delta_2 \in H^1(N_2';\mathbb{Z})$ and $N\in \mathbb{Z}^+$, such that for any $n>N$, $\alpha_2=n\beta_2+\delta_2\in H^1(N_2';\mathbb{Z})$ is a fibered class. For the restricted homomorphism $\delta_2|_{A'}:A'\to \mathbb{Z}$, $\delta_1=\delta_2|_{A'}\circ \phi:\pi_1(N_1')\to \mathbb{Z}$ is a cohomology class in $H^1(N_1';\mathbb{Z})$. Since $p_1^*(\beta_1)$ is a fibered class of $N_1'$, by Lemma \ref{perturbation}, for $n$ large enough, $\alpha_1=np_1^*(\beta_1)+\delta_1$ is also a fibered class of $N_1'$, and its restriction on $A'$ is a nontrivial homomorphism to $\mathbb{Z}$.

So we have $$\alpha_1|_{A'}=np_1^*(\beta_1)|_{A'}+\delta_1|_{A'}=n\beta_2|_{A'}+\delta_2|_{A'}\circ \phi|_{A'}=n\beta_2|_{A'}+\delta_2|_{A'}=\alpha_2|_{A'}.$$ If $\alpha_1|_{A'}=\alpha_2|_{A'}$ are surjective homomorphisms to $\mathbb{Z}$, then we just take $N_1=N_1'$ and $N_2=N_2'$.

If $\alpha_1|_{A'}=\alpha_2|_{A'}$ is not surjective, since it is nontrivial, the image is $d\mathbb{Z}<\mathbb{Z}$ for some $d\geq 2$. Then we take cyclic covers $q_1:N_1\to N_1'$ and $q_2:N_2\to N_2'$ corresponding to kernels of $\pi_1(N_1')\xrightarrow{\alpha_1}\mathbb{Z}\to \mathbb{Z}/d\mathbb{Z}$ and $\pi_1(N_2')\xrightarrow{\alpha_2}\mathbb{Z}\to \mathbb{Z}/d\mathbb{Z}$ respectively. Then we have $A'<\pi_1(N_1)$ and $A'<\pi_1(N_2)$, while $q_1^*(\alpha_1)$ and $q_2^*(\alpha_2)$ are surjective homomorphisms from $\pi_1(N_1)$ and $\pi_1(N_2)$ to $d\mathbb{Z}$ respectively. So $\frac{1}{d}q_1^*(\alpha_1)\in H^1(N_1;\mathbb{Z})$ and $\frac{1}{d}q_2^*(\alpha_2)\in H^1(N_2;\mathbb{Z})$ are both primitive cohomology classes, and they have the same restriction on $A'$. By abusing notation, we use $\alpha_1$ and $\alpha_2$ to denote $\frac{1}{d}q_1^*(\alpha_1)$ and $\frac{1}{d}q_2^*(\alpha_2)$. Then $N_1$ and $N_2$ are desired finite covers of $M_1$ and $M_2$, with primitive fibered classes $\alpha_1\in H^1(N_1;\mathbb{Z})$ and $\alpha_2\in H^1(N_2;\mathbb{Z})$, such that $\alpha_1|_{A'}=\alpha_2|_{A'}$ are surjective homomorphisms to $\mathbb{Z}$.

Since $\pi_1(N_1)\cap A=\pi_1(N_2)\cap A=A'$, $\pi_1(N_1)*_{A'}\pi_1(N_2)$ is a nontrivial geometrically finite amalgamation of $\pi_1(N_1)$ and $\pi_1(N_2)$. There is an obvious homomorphism from $\pi_1(N_1)*_{A'}\pi_1(N_2)$ to $\pi_1(M_1)*_A\pi_1(M_1)$, and it is injective by the canonical form of elements in an amalgamation product.

Since $\alpha_1:\pi_1(N_1)\to \mathbb{Z}$ and $\alpha_2:\pi_1(N_2)\to \mathbb{Z}$ agree with each other on $A'=\pi_1(N_1)\cap \pi_1(N_2)<\pi_1(N_1)*_{A'}\pi_1(N_2)$, they induce a homomorphism $\alpha: \pi_1(N_1)*_{A'}\pi_1(N_2)\to \mathbb{Z}$.

Let $\Sigma_1$ and $\Sigma_2$ be connected fibered surfaces of $N_1$ and $N_2$ corresponding to primitive fibered classes $\alpha_1$ and $\alpha_2$ respectively, and let $K$ be the kernel of the surjective homomorphism $\alpha_1|_{A'}=\alpha_2|_{A'}:A'\to \mathbb{Z}$. Since $\alpha_1|_{A'}=\alpha_2|_{A'}:A'\to \mathbb{Z}$ is surjective, it is easy to see that the kernel of $\alpha: \pi_1(N_1)*_{A'}\pi_1(N_2)\to \mathbb{Z}$ is $\pi_1(\Sigma_1)*_{K}\pi_1(\Sigma_2)$.

Since $\Sigma_1$ and $\Sigma_2$ are compact surfaces, and the kernel $\pi_1(\Sigma_1)*_{K}\pi_1(\Sigma_2)$ is generated by $\pi_1(\Sigma_1)$ and $\pi_1(\Sigma_2)$, $\pi_1(\Sigma_1)*_{K}\pi_1(\Sigma_2)$ is finitely generated. So $\pi_1(N_1)*_{A'}\pi_1(N_2)$ is algebraically fibered.
\end{proof}

\begin{remark}
For a general geometrically finite subgroup $A'$ of a finite volume hyperbolic $3$-manifold group, the kernel $K$ may not be finitely generated. So $\pi_1(\Sigma_1)*_{K}\pi_1(\Sigma_2)$ may not be finitely presented.
\end{remark}

\subsection{Virtually fibered boundary slopes on cusps}\label{notsorelated}

In this subsection, we study virtually fibered boundary slopes on boundary components of cusped hyperbolic $3$-manifolds. It is not directly related with the proof of Theorem \ref{main}, but it naturally shows up in the study of matching fibered structures of $3$-manifolds along $\mathbb{Z}^2$ subgroups. Although the $\mathbb{Z}^2$-amalgamation case is dealt in \cite{Sun} by using the result in \cite{PW}, the following results are quite interesting by themselves.

\begin{definition}
Let $M$ be a cusped hyperbolic $3$-manifold, $T$ be a boundary component of $M$, and $a$ be a slope on $T$. We say that $a$ is a {\it virtually fibered boundary slope} if there exist a finite cover $M'$ of $M$, an elevation $T'$ of $T$ in $M'$, and a fibered structure of $M'$, such that the corresponding fibered surface in $M'$ intersects with $T'$ along (parallal copies of) an elevation of $a$ in $T'$.
\end{definition}

For the $A\cong \mathbb{Z}^2$ case of Theorem \ref{matchingfibering}, we basically just find virtually fibered boundary slopes on $M_1$ and $M_2$ such that they match with each other under the pasting. Moreover, the proof of Theorem \ref{matchingfibering} implies that, for any torus boundary component $T$ of a cusped hyperbolic $3$-manifold, the set of virtually fibered boundary slopes form an open dense subset of $PH^1(T;\mathbb{Q})$. Moreover, in the following proposition, we prove there are only finitely many slopes on $T$ that may not be virtually fibered boundary slopes.

\begin{proposition}\label{fiberedboundaryslope}
Let $M$ be a cusped finite volume hyperbolic $3$-manifold, and $T$ be a boundary component of $M$, then all but finitely many slopes on $T$ are virtually fibered boundary slopes.
\end{proposition}

\begin{proof}
Since $\pi_1(T)<\pi_1(M)$ is a geometrically finite subgroup, by Theorem \ref{virtualretract}, there is a finite cover $M''$ of $M$ such that $\pi_1(T)<\pi_1(M'')$ and there is a retract homomorphism $\pi_1(M'')\to \pi_1(T)$. In particular, the torus $T$ lifts to $M''$ and the homomorphism $H_1(T;\mathbb{Z})\to H_1(M'';\mathbb{Z})$ induced by inclusion is injective.

By Theorem \ref{quasifiber}, there is a finite cover $p:M'\to M''$, such that for any nontrivial $\alpha\in H^1(M'';\mathbb{Z})$, $p^*(\alpha)\in H^1(M';\mathbb{Z})$ is a quasi-fibered class. Take any elevation $T'$ of $T$ in $M'$ and we will prove that all but finitely many slopes on $T'$ are boundary slopes of fibered surfaces of $M'$.

Let $i$ be the inclusion map $i:T'\to M'$. To show that a slope $a$ on $T'$ is the boundary slope of a fibered surface, we need only to show that there is a fibered class $\alpha\in H^1(M';\mathbb{Z})$, such that $i^*(\alpha)\in H^1(T';\mathbb{Z})$ is a nonzero multiple of the dual of $a$. Actually, since $i^*:H^1(M';\mathbb{R})\to H^1(T';\mathbb{R})$ is represented by an integer entry matrix, it suffices to show that there is a fibered class $\alpha\in H^1(M';\mathbb{R})$, such that $i^*(\alpha)\in H^1(T';\mathbb{R})$ lies in the line containing the dual of $a$.

By the fact that $H_1(T;\mathbb{R})\to H_1(M'';\mathbb{R})$ is injective, $i_*:H_1(T';\mathbb{R})\to H_1(M';\mathbb{R})$ is also injective. So the dual homomorphisms $H^1(M'';\mathbb{R})\to H^1(T;\mathbb{R})$ and $i^*:H^1(M';\mathbb{R})\to H^1(T';\mathbb{R})$ are surjective.

For any slope $a$ on $T'$, let the line in $H^1(T';\mathbb{R})$ containing the dual of $a$ be denoted by $l_a$. Since $i^*:H^1(M';\mathbb{R})\to H^1(T';\mathbb{R})$ is surjective, $(i^*)^{-1}(l_a)$ is a codimension-$1$ hyperplane in $H^1(M';\mathbb{R})$ going through the origin. Since covering maps always induce injective homomorphisms on real coefficient cohomology, we can identify $H^1(M'';\mathbb{R})$ as a subspace of $H^1(M';\mathbb{R})$. Since the covering map induces an isomorphism between $H^1(T;\mathbb{R})$ and $H^1(T';\mathbb{R})$ and $H^1(M'';\mathbb{R})\to H^1(T;\mathbb{R})$ is surjective, we can see that $(i^*)^{-1}(l_a)\cap H^1(M'';\mathbb{R})$ is also a codimension-$1$ hyperplane in $H^1(M'';\mathbb{R})$, and different slopes on $T'$ correspond to different codimension-$1$ hyperplanes in $H^1(M'';\mathbb{R})$.

There are only finitely many codimension-$1$ hyperplanes in $H^1(M'';\mathbb{R})$ that do not intersect with top dimensional (open) faces of the Thurston norm unit ball of $M''$, and these hyperplanes give us finitely many possible exceptional slopes. For any slope $a$ such that the corresponding codimension-$1$ hyperplane $(i^*)^{-1}(l_a)\cap H^1(M'';\mathbb{R})$ intersects with a top dimensional open face $F$ of the Thurston norm unit ball of $M''$, it clearly does not contain the whole face $F$ by dimensional reason. Since $F$ lies in the closure of an (open) fibered face $F'$ of $M'$, we have $F'\cap (i^*)^{-1}(l_a)\ne \emptyset$. So there is a fibered class $\alpha\in H^1(M';\mathbb{R})$, such that $i^*(\alpha)\in l_a$, thus the slope $a$ in $T'$ is the boundary slope of a fibered surface in $M'$.
\end{proof}

Proposition \ref{fiberedboundaryslope} implies that there are only finitely many slopes on $T$ that may not be virtually fibered boundary slopes. However, we do not have any single example of slope that is known not to be a virtually fibered boudnary slope. So maybe it is not too optimistic to ask whether all slopes are virtually fibered boundary slopes.

We do not have an answer for a general hyperbolic $3$-manifold. However, since arithmetic hyperbolic manifolds have a lot of symmetries in their finite covers, these symmetries imply that those finitely many possible exceptional slopes are actually virtually fibered boundary slopes. So we have an affirmative answer for cusped arithmetic hyperbolic $3$-manifolds.

\begin{proposition}\label{arithmeticslope}
Let $M$ be a cusped arithmetic hyperbolic $3$-manifold, and $T$ be a boundary component of $M$, then all slopes on $T$ are virtually fibered boundary slopes.
\end{proposition}

\begin{proof}
By Proposition \ref{fiberedboundaryslope}, there are finitely many slopes $a_1,a_2,\cdots,a_n$ on $T$, such that all other slopes on $T$ are virtually fibered boundary slopes. We fix a slope $a=a_1$, and use the arithmetic property of $M$ to show that $a$ is actually a virtually fibered boundary slope.

Recall that, for any cusped arithmetic hyperbolic $3$-manifold $M$, $\pi_1(M)$ is commensurable with $PSL_2(\mathcal{O}_d)$ for some square free $d\in \mathbb{Z}^+$ (Theorem 8.2.3 of \cite{MR}). Here $\mathcal{O}_d$ is the ring of algebraic integers of field $\mathbb{Q}(\sqrt{-d}).$

We can identify $\pi_1(M)$ as a subgroup of $PSL_2(\mathbb{C})\cong \text{Isom}_+(\mathbb{H}^3)$. Up to conjugation, we have $$\text{Comm}(\pi_1(M))=\text{Comm}(PSL_2(\mathcal{O}_d))=PGL_2(\mathbb{Q}(\sqrt{-d})).$$ For any $\Gamma<\text{Isom}_+(\mathbb{H}^3)$, its commensurator is defined by  $$\text{Comm}(\Gamma)=\{g\in \text{Isom}_+(\mathbb{H}^3)\ |\ g\Gamma g^{-1}\cap \Gamma \text{\ is\ a\ finite\ index\ subgroup\ of\ both\ }\Gamma \text{\ and\ } g\Gamma g^{-1}\}.$$

We assume that the cusp $T$ corresponds to a parabolic fixed point $\infty \in \mathbb{C}\cup\{\infty\}=\partial \mathbb{H}^3$ (under the upper half-space model). Then $\pi_1(T)$ is a subgroup of the parabolic stabilizer of $\infty$ in $PGL_2(\mathbb{Q}(\sqrt{-d}))$: $$\{
\begin{pmatrix}
  1 & r \\
  0 & 1
\end{pmatrix}\ |\ r\in \mathbb{Q}(\sqrt{-d})\}.$$

For any $\gamma=\begin{pmatrix}
  1 & r \\
  0 & 1
\end{pmatrix}\in \pi_1(T)$ with $r=a+b\sqrt{-d}$ and $a,b\in \mathbb{Q}$, we use $\text{arg}(a+b\sqrt{-d})$ as a coordinate of $\gamma$.

Suppose that the (possibly exceptional) slope $a$ corresponds to a parabolic element $h=\begin{pmatrix}
  1 & p+q\sqrt{-d}\\
  0 & 1
\end{pmatrix}$. Since there are only finitely many possibly exceptional slopes on $T$, there exists $\epsilon >0$ such that all slopes on $T$ with coordinate in $$(\text{arg}(p+q\sqrt{-d}),\text{arg}(p+q\sqrt{-d})+\epsilon)$$ are virtually fibered boundary slopes.

Then for any $s,r\in \mathbb{Q}$, with $\text{arg}(s+r\sqrt{-d})\in (0,\frac{\epsilon}{2})$, we consider $$g=\begin{pmatrix}
  s+r\sqrt{-d} & 0\\
  0 & (s+r\sqrt{-d})^{-1}
\end{pmatrix}\in PGL_2(\mathbb{Q}(\sqrt{-d}))=\text{Comm}(\pi_1(M)).$$ So $\pi_1(M)\cap g\pi_1(M)g^{-1}$ is a finite index subgroup of both $\pi_1(M)$ and $g\pi_1(M)g^{-1}$. Let $M'$ be the finite cover of $M$ corresponding to $\pi_1(M)\cap g\pi_1(M)g^{-1}$, then $M'$ covers $M$ in two different ways.

We consider $M$ and $M'$ as quotients of $\mathbb{H}^3$ by $M=\mathbb{H}^3/\pi_1(M)$ and $M'=\mathbb{H}^3/\pi_1(M)\cap g\pi_1(M)g^{-1}$. There are two covering maps $$p,p':M'=\mathbb{H}^3/\pi_1(M)\cap g\pi_1(M)g^{-1} \to M=\mathbb{H}^3/\pi_1(M)$$ defined by $p(\overline{x})=\overline{x}$ and $p'(\overline{x})=\overline{g^{-1}x}$. Then the induced homomorphisms on fundamental groups $p_*,p'_*:\pi_1(M')=\pi_1(M)\cap g\pi_1(M)g^{-1}\to \pi_1(M)$ are given by $p_*(h)=h$ and $p'_*(h)=g^{-1}hg$ respectively.

Let $T'$ be the boundary torus of $M'$ corresponding to $\infty \in \partial \mathbb{H}^3$. For the slope $a$ on $T$, it corresponds to $h=\begin{pmatrix}
  1 & p+q\sqrt{-d}\\
  0 & 1
\end{pmatrix}\in \pi_1(T)<\pi_1(M)$. Then under the covering map $p':M'\to M$, the elevation slope $a'$ of $a$ on $T'$ corresponds to $$gh^ng^{-1}=\begin{pmatrix}
  1 & n(s+r\sqrt{-d})^2(p+q\sqrt{-d})\\
  0 & 1
\end{pmatrix}\in \pi_1(T')<\pi_1(M')$$ for some $n\in \mathbb{Z}^+$. Under the covering map $p:M'\to M$, the projection of the slope $a'$ on $T$ still corresponds to matrix $\begin{pmatrix}
  1 & n(s+r\sqrt{-d})^2(p+q\sqrt{-d})\\
  0 & 1
\end{pmatrix}$, which has coordinate $$\text{arg}\big(n(s+r\sqrt{-d})^2(p+q\sqrt{-d})\big)\in (\text{arg}(p+q\sqrt{-d}),\text{arg}(p+q\sqrt{-d})+\epsilon),$$ since $\text{arg}(s+r\sqrt{-d})\in (0,\frac{\epsilon}{2})$.

Since we have assumed that all slopes on $T$ with coordinate in $(\text{arg}(p+q\sqrt{-d}), \text{arg}(p+q\sqrt{-d})+\epsilon)$ are virtually fibered boundary slopes, $a'$ is a virtually fibered boundary slope on $T'\subset \partial M'$ (via the covering map $p$). Since the covering map $p':M'\to M$ maps $a'$ to $a$, $a$ is a virtually fibered boundary slope on $T\subset \partial M$.
\end{proof}

\section{Construction of nonseparable subgroups}\label{nonseparablesubgroup}

In this section, for a nontrivial geometrically finite amalgamation $\pi_1(M_1)*_{A}\pi_1(M_2)$ with an algebraically fibered structure, we construct a nonseparable subgroup of it (we still use the notation $\pi_1(M_1)*_{A}\pi_1(M_2)$, instead of $\pi_1(N_1)*_{A'}\pi_1(N_2)$ as in Theorem \ref{matchingfibering}).

The first step is to find a further subgroup of $\pi_1(M_1)*_{A}\pi_1(M_2)$ (denoted by $\pi_1(N_1)*_{A',A''}\pi_1(N_2)$). This subgroup has an induced graph of group structure with two vertices and two edges, and also has an induced algebraically fibered structure.

The remaining part of the construction is more topological. We first construct a space $X$ that has a graph of space structure with $\pi_1(X)\cong \pi_1(N_1)*_{A',A''}\pi_1(N_2)$. Then we use the cycle in the dual graph of $X$ to construct a compact $2$-complex $Z$ and a map $f:Z\to X$, by pasting fibered surfaces in vertex pieces of $X$ together carefully. Then we show that $f_*(\pi_1(Z))<\pi_1(X)$ is not separable, by assuming it is separable and using Scott's topological interpretation of separability (\cite{Sc}) to get a contradiction.

\subsection{A further subgroup of algebraically fibered $\pi_1(M_1)*_A\pi_1(M_2)$}

In this subsection, for a nontrivial geometrically finite amalgamation $\pi_1(M_1)*_A\pi_1(M_2)$ that has an algebraically fibered structure, we construct a subgroup such that it fits into the ideal model for constructing nonseparable subgroups in \cite{Sun}.

\begin{proposition}\label{2V2E}
Suppose that $\pi_1(M_1)*_A\pi_1(M_2)$ is a nontrivial geometrically finite amalgamation of two finite volume hyperbolic $3$-manifold groups, which has an algebraically fibered structure and satisfies conditions in Theorem \ref{matchingfibering}. Then it has a subgroup $\pi_1(N_1)*_{A', A''}\pi_1(N_2)$ such that the following conditions hold.
\begin{enumerate}
  \item $N_1$, $N_2$ are finite covers of $M_1$, $M_2$ respectively.
  \item There are two nontrivial groups $A',A''$ and four injective homomorphisms $$\ \ \ \ \ \ \ \ \ \ i_1':A'\to \pi_1(N_1),\ i_2':A'\to \pi_1(N_2),\ i_1'':A''\to \pi_1(N_1),\ i_2'':A''\to \pi_1(N_2),$$ such that their images in $\pi_1(N_1)$ and $\pi_1(N_2)$ are geometrically finite and disjoint from each other except at the identity. Then $\pi_1(N_1)*_{A', A''}\pi_1(N_2)$ is isomorphic to the group with a graph of group structure induced by these four injective homomorphisms.
  \item There are two elements $g_1\in \pi_1(M_1)-\pi_1(N_1)$ and $\ g_2\in \pi_1(M_2)-\pi_1(N_2)$, such that $A'=A\cap \pi_1(N_1)$ and $A''=g_1Ag_1^{-1}\cap \pi_1(N_1)$ hold in $\pi_1(M_1)$; while $A'=A\cap \pi_1(N_2)$ and $A''=g_2Ag_2^{-1}\cap \pi_1(N_2)$ hold in $\pi_1(M_2)$. Here we identify $A'$ and $A''$ with their images in $\pi_1(N_1)$ and $\pi_1(N_2)$.
  \item There are fibered classes $\beta_1\in H^1(N_1;\mathbb{Z})$ and $\beta_2\in H^1(N_2;\mathbb{Z})$ such that $\beta_1|_{A'}=\beta_2|_{A'}$ and $\beta_1|_{A''}=\beta_2|_{A''}$, and they are all surjective homomorphisms to $\mathbb{Z}$.
  \item The homomorphism $H_1(A'* A'';\mathbb{Z})\to H_1(N_1;\mathbb{Z})$ induced by the inclusion is injective, and the image is a retraction of $H_1(N_1;\mathbb{Z})$.
\end{enumerate}
\end{proposition}

The proof is similar to the proof of Lemma 4.5 of \cite{Sun}.

\begin{proof}
Since $A<\pi_1(M_1)$ is a geometrically finite subgroup, its limit set is a proper closed subset of $S^2_{\infty}=\partial \mathbb{H}^3$. Take a loxodromic element $g_1\in\pi_1(M_1)-\{e\}$, such that both of its limit points do not lie in the limit set of $A$. For a large enough $n_1\in \mathbb{Z}^+$, the obvious homomorphism from $A*g_1^{n_1}Ag_1^{-n_1}$ to $\pi_1(M)$ is injective, and the image is geometrically finite. For simplicity, we denote the image by $A*g_1Ag_1^{-1}$.

Similarly, there is also an element $g_2\in \pi_1(M_2)-\{e\}$ such that the subgroup of $\pi_1(M_2)$ generated by $A$ and $g_2Ag_2^{-1}$ is isomorphic to $A*g_2Ag_2^{-1}$, and it is geometrically finite.

Now we apply LERFness of finite volume hyperbolic $3$-manifold groups to $A*g_1Ag_1^{-1}<\pi_1(M_1)$ and the virtual retract property of $A*g_1Ag_1^{-1}<\pi_1(M_1)$. Then there exist a finite cover $p_1:N_1\to M_1$, such that $g_1\not \in \pi_1(N_1)$, $A*g_1Ag_1^{-1}<\pi_1(N_1)$, and $A*g_1Ag_1^{-1}$ is a retraction of $\pi_1(N_1)$. By the same construction, we get a finite cover $p_2:N_2\to M_2$ with the same property. Actually, since we do not need condition (5) for $N_2$, a simpler construction works.

Then we can get the desired group $\pi_1(N_1)*_{A',A''}\pi_1(N_2)$ by identifying $A<\pi_1(N_1)$ with $A<\pi_1(N_2)$, and identifying $g_1Ag_1^{-1}<\pi_1(N_1)$ with $g_2Ag_2^{-1}<\pi_1(N_2)$. So actually $A'\cong A''\cong A$.

Then it is easy to see that conditions (1), (2), (3) in the proposition hold.

We take $\beta_1=p_1^*(\alpha_1)\in H^1(N_1;\mathbb{Z})$ and $\beta_2=p_2^*(\alpha_2)\in H^1(N_2;\mathbb{Z})$. Since $A'$ and $A''$ are just conjugations of $A$, condition (4) holds.

The fact that $A*g_1Ag_1^{-1}=A'*A''$ is a retraction of $\pi_1(N_1)$ implies that $H_1(A'*A'';\mathbb{Z})$ is a retraction of $H_1(N_1;\mathbb{Z})$. So the inclusion induces an injective homomorphism $H_1(A'*A'';\mathbb{Z})\to H_1(N_1;\mathbb{Z})$, with the image being a retraction of $H_1(N_1;\mathbb{Z})$. So condition (5) holds.

It remains to show that $\pi_1(N_1)*_{A',A''}\pi_1(N_2)$ is isomorphic to a subgroup of $\pi_1(M_1)*_A\pi_1(M_2)$, i.e. the obvious homomorphism is injective.

By van Kampen theorem, $\pi_1(N_1)*_{A',A''}\pi_1(N_2)$ is isomorphic to $$\langle \pi_1(N_1),\pi_1(N_2),t\ |\  i_1'(a')=i_2'(a'), i_1''(a'')=ti_2''(a'')t^{-1} \text{\ for\ any\ } a'\in A', a''\in A''\rangle.$$ Recall that $A'\cong A'' \cong A$, with $i_1''(a)=g_1i_1'(a)g_1^{-1}$ and $i_2''(a)=g_2i_2'(a)g_2^{-1}$ for $g_1\in \pi_1(M_1)-\pi_1(N_1)$ and $g_2\in \pi_1(M_2)-\pi_1(N_2)$.

Then each element $\sigma$ in $\pi_1(N_1)*_{A',A''}\pi_1(N_2)$ can be written in the form of $$\sigma=t^{k_1}h_1t^{k_2}h_2\cdots t^{k_n} h_n t^{k_{n+1}}.$$ Here $k_1,k_{n+1}\in \mathbb{Z}$, $k_2,\cdots, k_n\in \mathbb{Z}-\{0\}$, and each $h_i$ is a nontrivial product of elements in $\pi_1(N_1)$ and $\pi_1(N_2)$. Moreover, we also have the following conditions for $h_i$:
\begin{itemize}
  \item If $h_i$ is a product of more than one terms, then each term does not lie in $i_1'(A')<\pi_1(N_1)$ or $i_2'(A')<\pi_1(N_2)$.
  \item If $h_i$ is just (the product of) one element in $\pi_1(N_1)$ or $\pi_1(N_2)$, then
  \begin{itemize}
    \item if $h_i\in i_1''(A'')<\pi_1(N_1)$, then either $k_i\geq 0$ of $k_{i+1}\leq 0$.
    \item if $h_i\in i_2''(A'')<\pi_1(N_2)$, then either $k_i\leq 0$ of $k_{i+1}\geq 0$.
  \end{itemize}
\end{itemize}

Then the "obvious" homomorphism $$j:\pi_1(N_1)*_{A',A''}\pi_1(N_2)\to \pi_1(M_1)*_A\pi_1(M_2)$$ is defined by $$j(n_1)=n_1 \text{\ if\ } n_1\in \pi_1(N_1),\  j(n_2)=n_2 \text{\ if\ } n_2\in \pi_1(N_2),\ j(t)=g_1g_2^{-1}.$$ By the above form of elements in $\pi_1(N_1)*_{A',A''}\pi_1(N_2)$, the choice of $g_1$ and $g_2$, and the canonical form of elements in an amalgamation, it is routine to check that $j$ is injective. So $\pi_1(N_1)*_{A',A''}\pi_1(N_2)$ is isomorphic to a subgroup of $\pi_1(M_1)*_A\pi_1(M_2)$.
\end{proof}

\subsection{Realize $\pi_1(N_1)*_{A', A''}\pi_1(N_2)$ as the fundamental group of a space}\label{space}

All proofs in the previous part of this paper are in algebraic fashion, since we only worked on the group level and did not realize $\pi_1(M_1)*_{A}\pi_1(M_2)$ and $\pi_1(N_1)*_{A', A''}\pi_1(N_2)$ as fundamental groups of topological spaces. In the following part of this paper, when we construct nonseparable subgroups, we need to realize $\pi_1(N_1)*_{A', A''}\pi_1(N_2)$ as the fundamental group of a topological space. In this subsection, we construct this topological space, and develop some definition for the convenience of further constructions.

For finite volume hyperbolic $3$-manifolds $N_1,N_2$ as in Proposition \ref{2V2E}, we take covering spaces $p_1':\tilde{N}_1'\to N_1$ and $p_2':\tilde{N}_2'\to N_2$ corresponding to $A'<\pi_1(N_1)$ and $A'<\pi_1(N_2)$ respectively. Although having isomorphic fundamental groups, $\tilde{N}_1'$ and $\tilde{N}_2'$ may not be homeomorphic to each other. However, since hyperbolic manifolds are all Eilenberg-Maclane spaces ($K(\pi,1)$), there exists a homotopy equivalence $f':\tilde{N}_1'\to \tilde{N}_2'$. Similarly, for covering spaces $p_1'':\tilde{N}_1''\to N_1$ and $p_2'':\tilde{N}_2''\to N_2$ corresponding to $A''<\pi_1(N_1)$ and $A''<\pi_1(N_2)$ respectively, we also have a homotopy equivalence $f'':\tilde{N}_1''\to \tilde{N}_2''$. For most part of this paper, the readers can just think $f'$ and $f''$ as homeomorphisms.

Now we construct a space $X$ from $N_1\sqcup N_2\sqcup \tilde{N}_1'\times I \sqcup \tilde{N}_1''\times I$ by the following pasting maps: $$p_1': \tilde{N}_1'\times \{0\}\to N_1,\ p_2'\circ f':\tilde{N}_1'\times \{1\}\to N_2,$$
$$p_1'': \tilde{N}_1''\times \{0\}\to N_1,\ p_2''\circ f'':\tilde{N}_1''\times \{1\}\to N_2.$$ Here we apply $p_1',p_1'',f',f''$ to slices of $\tilde{N}_1'\times I$ and $\tilde{N}_1''\times I$ by restricting to the first coordinate.

By van Kampen theorem, $\pi_1(X)$ is isomorphic to $\pi_1(N_1)*_{A', A''}\pi_1(N_2)$.

We also need to define a notion of immersed objects in $X$ that give desired nonseparable subgroups. This will play the role of properly immersed surfaces in mixed $3$-manifolds, and "properly immersed singular surfaces" in $N_1\cup_{c_1\cup c_2}N_2$ as in \cite{Sun}. The main difference is that the immersed object we will construct is not $\pi_1$-injective.

\begin{definition}
For the space $X$ as above, a {\it generalized immersed surface} in $X$ is a pair $(Z,f)$ where $Z$ is a connected compact $2$-complex and $f:Z\to X$ is a map such that the following conditions hold.
\begin{enumerate}
  \item $Z$ is constructed by pasting finitely many intervals $\sqcup I_j$ to finitely many compact oriented surfaces $\sqcup S_i$, by identifying all end points of intervals with distinct points in surfaces.
  \item The $f$-image of each $S_i$ either entirely lies in $N_1$ or entirely lies in $N_2$. Moreover, $f|_{S_i}:S_i\to N_k$ is a $\pi_1$-injective immersion for the corresponding $k\in \{1,2\}$.
  \item The $f$-image of each $I_j$ either entirely lies in $\tilde{N}_1'\times I$ or entirely lies in $\tilde{N}_2'\times I$. Moreover, let $p_k:\tilde{N}_k'\times I\to I$ be the projection to the second factor, then $p_k\circ f|_{I_j}:I_j\to I$ is the identity map on interval, for the corresponding $k\in \{1,2\}$.
\end{enumerate}
\end{definition}

For a generalized immersed surface $(Z,f)$ in $X$, the induced homomorphism $f_*:\pi_1(Z)\to \pi_1(X)$ may not be injective.

\subsection{Construct a generalized immersed surface that carries a nonseparable subgroup}

In this subsection, for the space $X$ as above with fundamental group $\pi_1(N_1)*_{A', A''}\pi_1(N_2)$, we construct a generalized immersed surface $(Z,f)$ in $X$, such that $f_*(\pi_1(Z))$ is the candidate of a nonseparable subgroup of $\pi_1(X)$.

The construction of this nonseparable subgroup is similar to the construction in \cite{Sun}, and the proof of nonseparability essentially follows the idea in \cite{Liu} and \cite{RW}. The only difference is that we do not require the nonseparable subgroup is carried by a $\pi_1$-injective map, and the following statement on the construction of $(Z,f)$ is more complicated.

\begin{proposition}\label{constructnonseparable}
Let $\pi_1(N_1)*_{A',A''}\pi_1(N_2)$ be a group that has a graph of group structure, an algebraically fibered structure, and satisfies all conditions in Proposition \ref{2V2E}. Let $X$ be the topological space with fundamental group isomorphic to $\pi_1(N_1)*_{A',A''}\pi_1(N_2)$ constructed in subsection \ref{space}, then there exists a generalized immersed surface $(Z,f)$ such that the following conditions hold.
\begin{enumerate}
  \item $Z$ is a quotient space of $2+2n$ compact connected oriented surfaces $\{S_{1,i}\}_{i=1}^2$, $\{S_{2,j}\}_{j=1}^{2n}$ and $4n$ intervals $\{I_k\}_{k=1}^{4n}$ (with $n\geq 2$), by identifying all end points of intervals with points in surfaces. Moreover, all surfaces $\{S_{2,j}\}_{j=1}^{2n}$ are homeomorphic to each other.
  \item For each interval $I_k\subset Z$, one of its end point lies in some $S_{1,i}$ and the other end point lies in some $S_{2,j}$. There are $2n+1$ intervals connecting $S_{1,1}$ to surfaces in $\{S_{2,j}\}_{j=1}^{2n}$, and $2n-1$ intervals connecting $S_{1,2}$ to surfaces in $\{S_{2,j}\}_{j=1}^{2n}$. Moreover, all these end points have different images in $X$.
  \item There are two intervals (say $I_1$ and $I_2$) connecting $S_{1,1}$ to $S_{2,1}$.
  \item For each $S_{1,i}$ and $S_{2,j}$, the restriction of $f$ on this surface is an embedding into $N_1$ and $N_2$ respectively, and the image is a fibered surface. Moreover, the images of $\{S_{2,j}\}_{j=1}^{2n}$ are $2n$ parallel copies of an oriented fibered surface in $N_2$.
  \item Let $\gamma_{1,1},\gamma_{1,2}\in H^1(N_1;\mathbb{Z})$ be fibered classes of $N_1$ corresponding to fibered surfaces $S_{1,1},S_{1,2}$ respectively, and $\gamma_2\in H^1(N_2;\mathbb{Z})$ be the fibered class of $N_2$ corresponding to one of the parallel fibered surfaces $S_{2,j}$. Then we have $$\gamma_{1,1}|_{A'}=(n+1)\gamma_2|_{A'},\ \gamma_{1,1}|_{A''}=n\gamma_2|_{A''},$$
      $$\gamma_{1,2}|_{A'}=(n-1)\gamma_2|_{A'},\ \gamma_{1,2}|_{A''}=n\gamma_2|_{A''},$$ as homomorphisms to $\mathbb{Z}$. Moreover, $\gamma_2|_{A'}$ and $\gamma_2|_{A''}$ are surjective homomorphisms.
  \item There exist closed embedded oriented circles $c'$ in $\tilde{N}_1'$ and $c''$ in $\tilde{N}_1''$ such that the following hold. The algebraic intersection numbers of $p_1'(c')$ with $S_{1,1}$ and $S_{1,2}$ in $N_1$ are $n+1$ and $n-1$ respectively, the algebraic intersection numbers of $p_1''(c'')$ with $S_{1,1}$ and $S_{1,2}$ in $N_1$ are both $n$, and the algebraic intersection numbers of $p_2'\circ f'(c')$ and $p_2''\circ f''(c'')$ with each $S_{2,j}$ in $N_2$ are exactly $1$.
  \item The following statement holds for all triples $$(p_1'(c'),S_{1,1},n+1),\ (p_1'(c'),S_{1,2},n-1),\ (p_1''(c''),S_{1,1},n),$$ $$(p_1''(c''),S_{1,2},n),\ (p_2'\circ f'(c'),S_{2,j},1),\ (p_2''\circ f''(c''),S_{2,j},1).$$ We only state it for the triple $(p_1'(c'),S_{1,1},n+1)$, and the statements for other triples are similar.

      There exist $n+1$ points $a_1,a_2,\cdots,a_{n+1}$ in $c'\cap (p_1')^{-1}(S_{1,1})$ such that the following hold.
      \begin{enumerate}
      \item The points $a_1,a_2,\cdots,a_{n+1}$ follow the orientation of $c'$, and the local algebraic intersection number of $p_1'(c')$ and $S_{1,1}$ at each $a_i$ is $1$.
      \item We take the oriented subarc of $c'$ from $a_1$ to $a_i$, then slightly move it along the positive direction of $c'$ and get an oriented subarc $\rho_i\subset c'$ whose end points are away from $(p_1')^{-1}(S_{1,1})$. Then the algebraic intersection number of $p_1'(\rho_i)$ and $S_{1,1}$ is $i-1$.
      \item $p_1'(a_1),p_1'(a_2),\cdots,p_1'(a_{n+1})$ are $n+1$ points on $S_{1,1}$ that are identified with end points of intervals in $Z$. Moreover, $f$ maps these $n+1$ intervals to $\tilde{N}_1'\times [0,1]$.
      \end{enumerate}
  \item There are exactly $n+1$ intervals in $Z$ connecting $S_{1,1}$ to $\{S_{2,j}\}$ that are mapped to $\tilde{N}_1'\times [0,1]$, they give a one-to-one correspondence between $n+1$ points in $c'\times \{0\}$ ($a_1,a_2,\cdots,a_{n+1}$ as above) and $n+1$ points in $c'\times \{1\}$ such that this correspondence preserves the cyclic order on the oriented circle $c'$. Moreover, the $f$-image of these intervals lie in $c'\times [0,1]$, they are disjoint from each other, and their projections to $c'$ are embedded (possibly degenerate) subarcs of $c'$. The same statement holds for $n-1$ edges in $Z$ connecting $S_{1,2}$ to $\{S_{2,j}\}$ that are mapped to $\tilde{N}_1'\times [0,1]$, $n$ edges in $Z$ connecting $S_{1,1}$ to $\{S_{2,j}\}$ that are mapped to $\tilde{N}_1''\times [0,1]$, and $n$ edges in $Z$ connecting $S_{1,2}$ to $\{S_{2,j}\}$ that are mapped to $\tilde{N}_1''\times [0,1]$ .
\end{enumerate}
\end{proposition}

The readers should compare Proposition \ref{constructnonseparable} with Proposition 4.8 of \cite{Sun}. Although the statement of Proposition \ref{constructnonseparable} is more complicated than Proposition 4.8 of \cite{Sun}, it just follows the same idea. Since we have a more complicated space $X$ in the current situation, the statement gets more complicated.

\begin{proof}
By Proposition \ref{2V2E} (4), there are fibered classes $\beta_1\in H^1(N_1;\mathbb{Z})$ and $\beta_2\in H^1(N_2;\mathbb{Z})$ such that $$\beta_1|_{A'}=\beta_2|_{A'},\ \beta_1|_{A''}=\beta_2|_{A''}$$ hold, and they are both surjective homomorphisms to $\mathbb{Z}$.

By Proposition \ref{2V2E} (5), the homomorphism $H_1(A'*A'';\mathbb{Z})\to H_1(N_1;\mathbb{Z})$ induced by inclusion is injective, with a retract homomorphism $\phi: H_1(N_1;\mathbb{Z})\to H_1(A'*A'';\mathbb{Z})$. So for the homomorphism $\delta:A'*A''\to \mathbb{Z}$ (equivalently $H_1(A'*A'';\mathbb{Z})\to \mathbb{Z}$) defined by $\delta|_{A'}=\beta_1|_{A'},\ \delta|_{A''}=0$, the composition $\gamma=\delta\circ \phi:H_1(N_1;\mathbb{Z})\to \mathbb{Z}$ gives a cohomology class $\gamma\in H^1(N_1;\mathbb{Z})$.

Since $\beta_1$ is a fibered class of $N_1$, by Lemma \ref{perturbation}, for large enough $n\in \mathbb{Z}^+$, $$\gamma_{1,1}=n\beta_1+\gamma,\ \gamma_{1,2}=n\beta_1-\gamma\in H^1(N_1;\mathbb{Z})$$ are both fibered classes of $N_1$. Then we take $\gamma_2=\beta_2$ as the desired fibered class of $N_2$. Since $\gamma|_{A'}=\beta_1|_{A'}$ and $\gamma|_{A''}=0$, it is easy to check that condition (5) holds. For example, we have $$\gamma_{1,1}|_{A'}=n\beta_1|_{A'}+\gamma|_{A'}=(n+1)\beta_1|_{A'}=(n+1)\beta_2|_{A'}=(n+1)\gamma_2|_{A'}.$$

Since $\gamma_2=\beta_2$, $\gamma_2$ is a primitive class in $H^1(N_2;\mathbb{Z})$, and its restrictions on $A'$ and $A''$ are surjective homomorphisms to $\mathbb{Z}$. Since $\gamma_{1,1}|_{A'}$ has image $(n+1)\mathbb{Z}$ and $\gamma_{1,1}|_{A''}$ has image $n\mathbb{Z}$, $\gamma_{1,1}$ is a primitive class in $H^1(N_1;\mathbb{Z})$. The same argument implies $\gamma_{1,2}$ is also a primitive class.

Then we take connected oriented fibered surfaces $S_{1,1}$ and $S_{1,2}$ in $N_1$ corresponding to fibered classes $\gamma_{1,1}$ and $\gamma_{1,2}$ respectively, take $2n$ parallel copies of the connected oriented fibered surface in $N_2$ corresponding to $\gamma_2$, and denote them by $\{S_{2,j}\}_{j=1}^{2n}$. The construction of these fibered surfaces satisfies condition (4).

Since $\gamma_2|_{A'}:A'\to \mathbb{Z}$ is surjective and $\pi_1(\tilde{N}_1')\cong A'$, there exists a closed embedded oriented circle $c'$ in $\tilde{N}_1'$, such that $\gamma_2([c'])=1$. Then we have $\gamma_{1,1}([c'])=(n+1)\gamma_2([c'])=n+1$ and $\gamma_{1,2}([c'])=(n-1)\gamma_2([c'])=n-1$. Similarly, there exists a closed embedded oriented circle $c''$ in $\tilde{N}_1''$, such that $\gamma_2([c''])=1$, and $\gamma_{1,1}([c''])=\gamma_{1,2}([c''])=n$. These equalities imply that condition (6) holds. We can also homotopy $c'$ and $c''$ such that their images in $N_1$ and $N_2$ have general position with $S_{1,1},S_{1,2}$ and $S_{2,j}$.

Now we consider condition (7). We only work on the triple $(p_1'(c'),S_{1,1},n+1)$, and the same argument works for all other triples. Actually, the requirements in condition (7) instruct us how to choose the points $a_1,a_2,\cdots,a_{n+1}$. Since we supposed that $p_1'(c')$ intersects with $S_{1,1}$ transversely, each point in $p_1(c')\cap S_{1,1}$ has local algebraic intersection number $\pm 1$, and their sum is $\gamma_{1,1}([c'])=n+1$. We start with a point $a_1\in c'$ with local algebraic intersection number $1$, then follow the orientation of $c'$ and sum algebraic intersection numbers of points in $c'\cap (p_1')^{-1}(S_{1,1})$ we have visited. Let $a_i$ be the first point we get total algebraic intersection number $i$, then it is obvious that we have local algebraic intersection number $1$ at each $a_i$, and desired properties in condition (7) (a) and (7) (b) hold.

Condition (7) (a) and (7) (b) provides us $2n=(n+1)+(n-1)$ points in $c'\cap (p_1')^{-1}(S_{1,1}\cup S_{1,2})$, $2n=n+n$ points in $c''\cap (p_1'')^{-1}(S_{1,1}\cup S_{1,2})$, $2n=1\times 2n$ points in $c'\cap (p_2'\circ f')^{-1}(\cup_{j=1}^{2n}S_{2,j})$ and $2n=1\times 2n$ points in $c''\cap (p_2''\circ f'')^{-1}(\cup_{j=1}^{2n}S_{2,j})$. These points will be pasted with end points of intervals in $Z$, and we will see that these numbers of points fulfill the requirements in condition (2).

Then we pair the $2n$ points in $c'\cap (p_1')^{-1}(S_{1,1}\cup S_{1,2})$ with the $2n$ points in $c'\cap (p_2'\circ f')^{-1}(\cup_{j=1}^{2n}S_{2,j})$, and pair the $2n$ points in $c''\cap (p_1')^{-1}(S_{1,1}\cup S_{1,2})$ with the $2n$ points in $c''\cap (p_2''\circ f'')^{-1}(\cup_{j=1}^{2n}S_{2,j})$, such that these pairings preserve cyclic orders on $c'$ and $c''$. To fulfill condition (3), we need to pair one point in $c'\cap (p_1')^{-1}(S_{1,1})$ with the point in $c'\cap (p_2'\circ f')^{-1}(S_{2,1})$, and one point in $c''\cap (p_1'')^{-1}(S_{1,1})$ with the point in $c''\cap (p_2''\circ f'')^{-1}(S_{2,1})$.

Now we construct the map on intervals $\{I_k\}_{k=1}^{4n}$. For each pair of points in $c'$ and $c''$ that are identified by the above pairing, we construct a map on an interval to connect these two points. For example, if these two points are $q_1\in c'\cap (p_1')^{-1}(S_{1,1})$ and $q_2\in c'\cap (p_2'\circ f')^{-1}(S_{2,1})$, then we use an embedded subarc $\delta$ of $c'$ from $q_1$ to $q_2$ to connect these two points. Then we get a map from $[0,1]$ to $\tilde{N}_1'\times [0,1]$ defined by $t\to (\delta(t),t)$ connecting $(q_1,0)$ and $(q_2,1)$. This construction on intervals satisfies condition (2), (7) (c) and (8). We can also make sure that these intervals are disjoint from each other in $c'\times [0,1]$.

In the above process, we have $2n$ maps from $[0,1]$ to $\tilde{N}_1'\times I$, and $2n$ maps from $[0,1]$ to $\tilde{N}_1''\times I$. We paste their endpoints with the corresponding points in $(S_{1,1}\sqcup S_{1,2})\sqcup(\sqcup_{j=1}^{2n}S_{2,j})$, and get the desired $2$-complex $Z$. The map $f$ is already given during our construction. So we get a generalized immersed surface $(Z,f)$ that satisfies all desired conditions.

\end{proof}

\subsection{Construction of the covering space of $X$ corresponding to $f_*(\pi_1(Z))$}\label{complicated}

In this subsection, we figure out the topology of the covering space  $\hat{X}$ of $X$ corresponding to $f_*(\pi_1(Z))<\pi_1(X)$, and show that $f:Z\to X$ lifts to an embedding $\hat{f}:Z\hookrightarrow \hat{X}$. In the next subsection, with the knowledge that $Z$ lifts to be embedded in $\hat{X}$, we suppose $f_*(\pi_1(Z))<\pi_1(X)$ is separable, then apply Scott's topological interpretation of separability (\cite{Sc}) and get a contradiction. The corresponding covering space in \cite{Sun} is easy to figure out, since the edge spaces in \cite{Sun} are very simple. Our current case is more complicated, so we give a more detailed construction of the covering space.

Let $\hat{N}_{1,1}$ and $\hat{N}_{1,2}$ be the infinite cyclic cover of $N_1$ corresponding to (the kernel of) $\gamma_{1,1}$ and $\gamma_{1,2}$ respectively, and let $\hat{N}_2$ be the infinite cyclic cover of $N_2$ corresponding to $\gamma_2$. Let $\hat{N}'_1$ be the infinite cyclic cover of $\tilde{N}_1'$ corresponding to $\gamma_{1,1}|_{A'}$, and $\hat{N}''_1$ be the infinite cyclic cover of $\tilde{N}_1''$ corresponding to $\gamma_{1,1}|_{A''}$. Similarly, let $\hat{N}'_2$ be the infinite cyclic cover of $\tilde{N}_2'$ corresponding to $\gamma_2|_{A'}$, and $\hat{N}''_2$ be the infinite cyclic cover of $\tilde{N}_2''$ corresponding to $\gamma_2|_{A''}$.

Note that $\hat{N}'_1$ is homeomorphic to the infinite cyclic cover of $\tilde{N}_1'$ corresponding to $\gamma_{1,2}|_{A'}$, since $\gamma_{1,1}|_{A'}$ and $\gamma_{1,2}|_{A'}$ are both nonzero multiples of $\gamma_2|_{A'}$, and they have the same kernel. Moreover, $\hat{N}'_1$ is homotopic equivalent to $\hat{N}'_2$, since $A'=\pi_1(\tilde{N}_1')$ is isomorphic to $A'=\pi_1(\tilde{N}_2')$ and $\gamma_{1,1}|_{A'}$ has the same kernel as $\gamma_2|_{A'}$. The same statement also holds for $\hat{N}''_1$.

Let $N_1^*=N_1\cup \tilde{N}_1'\times [0,1]\cup \tilde{N}_1''\times [0,1]$ be the subspace of $X$ corresponding to one vertex and two edges of the dual graph of $X$. Technically we should take $\tilde{N}_1'\times [0,1-\epsilon]$ and $\tilde{N}_1''\times [0,1-\epsilon]$, or saying that $X$ is a quotient space of $N_1^*$, but we abuse notation here. Note that $N_1\subset N_1^*$ is a deformation retract of $N_1^*$.

Let $\hat{N}_{1,1}^*$ and $\hat{N}_{1,2}^*$ be the infinite cyclic cover of $N_1^*$ corresponding to $\gamma_{1,1}:\pi_1(N_1)\to \mathbb{Z}$ and $\gamma_{1,2}:\pi_1(N_1)\to \mathbb{Z}$ respectively. Since $\pi_1(\tilde{N}_1')=A'$ and $\pi_1(\tilde{N}_1'')=A''$, $\gamma_{1,1}|_{A'}$ is an $(n+1)$-multiple of a primitive element in $H^1(A';\mathbb{Z})$ and  $\gamma_{1,1}|_{A''}$ is an $n$-multiple of a primitive element in $H^1(A'';\mathbb{Z})$, $\hat{N}_{1,1}^*$ is the union of $\hat{N}_{1,1}$, $n+1$ copies of $\hat{N}'_1\times [0,1]$ and $n$ copies of $\hat{N}''_1\times [0,1]$. Similarly, $\hat{N}_{1,2}^*$ is the union of $\hat{N}_{1,2}$, $n-1$ copies of $\hat{N}'_1\times [0,1]$ and $n$ copies of $\hat{N}''_1\times [0,1]$.

The space $Z$ has a graph of space structure. Its dual graph $G$ has $2+2n$ vertices, and they correspond to $\{S_{1,1},S_{1,2}\}\cup \{S_{2,j}\}_{j=1}^{2n}$. There are $4n$ edges in $G$ and each of them connects one vertex in $\{S_{1,1},S_{1,2}\}$ to one vertex in $\{S_{2,j}\}_{j=1}^{2n}$, and they correspond to intervals in $Z$. Each edge also has a marking in $\{1,2\}$, corresponding to whether $f$ maps this edge to $\tilde{N}_1'\times [0,1]$ or $\tilde{N}_1''\times [0,1]$ respectively.

Let $S_{1,1}^*$ be the union of $S_{1,1}$ and all its adjacent edges in $Z$, and $S_{1,2}^*$ be the union of $S_{1,2}$ and all its adjacent edges in $Z$. Then $f|_{S_{1,1}^*}:S_{1,1}^*\to N_1^*$ and $f|_{S_{1,2}^*}:S_{1,2}^*\to N_1^*$ have liftings $\hat{f}|_{S_{1,1}^*}:S_{1,1}^*\to \hat{N}_{1,1}^*$ and $\hat{f}|_{S_{1,2}^*}:S_{1,2}^*\to \hat{N}_{1,2}^*$ respectively. For $S_{1,1}^*$, it is clear that $\hat{f}|_{S_{1,1}^*}$ maps edges in $S_{1,1}^*$ with marking $1$ to copies of $\hat{N}_1'\times [0,1]$, and maps edges with marking $2$ to copies of $\hat{N}_1''\times [0,1]$, and the same statement holds for $\hat{f}|_{S_{1,2}^*}$.

Moreover, condition (7) in Proposition \ref{constructnonseparable} implies that the $n+1$ edges in $S_{1,1}^*$ with marking $1$ are mapped to $n+1$ different copies of $\hat{N}'_1\times [0,1]$ in $\hat{N}_{1,1}^*$ (see the proof of Proposition 4.10 of \cite{Sun}). The same statement also holds for the $n$ edges in $S_{1,1}^*$ with marking $2$, the $n-1$ edges in $S_{1,2}^*$ with marking $1$ and the $n$ edges in $S_{1,2}^*$ with marking $2$.

Then we take $2n$ copies of $\hat{N}_2$ and denote them by $\{\hat{N}_{2,j}\}_{j=1}^{2n}$, each endowed with a lifting of $S_{2,j}\subset N_2$ to $S_{2,j}\subset\hat{N}_{2,j}$, and two marked points in $S_{2,j}$ corresponding to end points of intervals in $Z$ that are pasted to $S_{2,j}$.

For each $S_{2,j}\subset Z$, there are two intervals in $Z$ connected with it. $\hat{f}|_{S_{1,1}^*}$ and $\hat{f}|_{S_{1,2}^*}$ map these two intervals to some copy of $\hat{N}_1'\times [0,1]$ and $\hat{N}_1''\times [0,1]$ in $\hat{N}_{1,1}^*$ or $\hat{N}_{1,2}^*$. Then we can paste the ends of $\hat{N}_1'\times [0,1]$ and $\hat{N}_1''\times [0,1]$ with second coordinate $1$ to $\hat{N}_{2,j}$ by maps $\hat{N}_1'\times \{1\}\to\hat{N}_2'\to \hat{N}_{2,j}$ and $\hat{N}_1''\times \{1\}\to\hat{N}_2''\to \hat{N}_{2,j}$ that send the end points of intervals in $\hat{N}_1'\times \{1\}$ and $\hat{N}_1''\times \{1\}$ to the corresponding two marked points in $S_{2,j}\subset \hat{N}_{2,j}$. In the composition $\hat{N}_1'\times \{1\}\to\hat{N}_2'\to \hat{N}_{2,j}$, the first map is a homotopy equivalence and the second one is a covering map. It is the lifting of $p_2'\circ f':\tilde{N}_1'\times \{1\}\to \tilde{N}_2'\to N_2$ to the corresponding infinite cyclic covers.

After pasting $2n$ copies of $\hat{N}'_1\times \{1\}$ and $2n$ copies of $\hat{N}_1''\times \{1\}$ in $\hat{N}_{1,1}^*$ and $\hat{N}_{1,2}^*$ with $2n$ copies of $\hat{N}_2$ by the above process, we get a space $Y$ with a graph of space structure such that its dual graph is isomorphic to the dual graph of $Z$. Moreover, we have an embedding $\hat{f}:Z\hookrightarrow Y$ that induces an isomorphism on the dual graphs.

Then we prove that the space $Y$ constructed above is the covering space of $X$ corresponding to the subgroup $f_*(\pi_1(Z))<\pi_1(X)$.

\begin{lemma}\label{topology}
Let $\hat{X}$ be the covering space of $X$ corresponding to the subgroup $f_*(\pi_1(Z))<\pi_1(X)$, then $\hat{X}$ is homeomorphic to the space $Y$ constructed above.

Moreover, $f:Z\to X$ lifts to an embedding $\hat{f}:Z\hookrightarrow \hat{X}$.
\end{lemma}

\begin{proof}
We first construct a covering map $\pi:Y\to X$, such that the embedding $\hat{f}:Z\hookrightarrow Y$ is a lifting of $f:Z\to X$, then show that $\hat{f}:Z\to Y$ is a $\pi_1$-surjective map. The existence of a lifting $\hat{f}:Z\to Y$ implies $f_*(\pi_1(Z))<\pi_1(Y)$, and $\pi_1$-surjectivity implies $f_*(\pi_1(Z))=\pi_1(Y)$. So $Y$ is homeomorphic to the covering space of $X$ corresponding to $f_*(\pi_1(Z))$, and $\hat{f}:Z\hookrightarrow Y$ is an embedded lifting of $f:Z\to X$.

The covering map $\pi:Y\to X$ is almost obvious. On $\hat{N}_{1,1}^*\subset Y$ and $\hat{N}_{1,2}^*\subset Y$, we take their covering maps to $N_1^*$ in the construction of $\hat{N}_{1,1}^*$ and $\hat{N}_{1,2}^*$. On each $\hat{N}_{2,j}\subset Y$, we take its covering map to $N_2$. Then we get a well-defined map $\pi:Y\to X$. The only thing we need to check is that $\pi$ is a local homeomorphism near each $\hat{N}_{2,j}$, since all the other points in $Y$ lie in the interior of $\hat{N}_{1,1}^*\subset Y$ or $\hat{N}_{1,2}^*\subset Y$. This local homeomorphism property is easy to check and we leave it to the readers.

The definition of $\hat{f}:Z\to Y$ immediately implies that $\pi\circ \hat{f}=f$, so $\hat{f}:Z\hookrightarrow Y$ is an embedded lifting of $f:Z\to X$.

Now we show that $\hat{f}:Z\to Y$ is $\pi_1$-surjective. At first, by the construction of $Y$, the restriction of $\hat{f}$ on each vertex space of $Z$ to the corresponding vertex piece of $Y$ induces an isomorphism on their fundamental groups. Since the fundamental group of each edge space of $Y$ is a subgroup of fundamental groups of its two adjacent vertex spaces, the van-Kampen theorem implies that $\pi_1(Y)$ is generated by the fundamental group of its vertex spaces, and the fundamental group of the dual graph of $Y$ (by choosing a "section" $s:G\to Y$ from the dual graph $G$ to $Y$). Since $\hat{f}:Z\to Y$ induces an isomorphism on the dual graph, we can assume the section $s:G\to Y$ factors through $\hat{f}:Z\to Y$. Then the groups of vertex spaces of $Y$ and the group of the dual graph of $Y$ (via a section) are both contained in $\hat{f}_*(\pi_1(Z))$. So $\hat{f}:Z\to Y$ is $\pi_1$-surjective.
\end{proof}

\subsection{Nonseparability of $f_*(\pi_1(Z))<\pi_1(X)$}

In this subsection, we prove that the subgroup $f_*(\pi_1(Z))<\pi_1(X)$ constructed in Proposition \ref{constructnonseparable} is not separable.

\begin{proposition}\label{proofnonseparable}
For the space $X$ constructed in subsection \ref{space} and the generalized immersed surface $(Z,f)$ constructed in Proposition \ref{constructnonseparable}, $f_*(\pi_1(Z))$ is a nonseparable subgroup of $\pi_1(X)=\pi_1(N_1)*_{A',A''}\pi_1(N_2)$.
\end{proposition}

The proof is similar to the proof of Proposition 4.10 of \cite{Sun}, and the idea goes back to \cite{Liu} and \cite{RW}. The only difference is that, in the current situation, $f:Z\to X$ is not a $\pi_1$-injective map. However, the subgroup $f_*(\pi_1(Z))<\pi_1(X)$ is still manageable, since we showed that $Z$ lifts to be embedded in $\hat{X}$ (Lemma \ref{complicated}).

\begin{proof}
We suppose that $f_*(\pi_1(Z))<\pi_1(X)$ is separable, and get a contradiction.

Recall that we have constructed the covering space $\hat{X}$ of $X$ corresponding to $f_*(\pi_1(Z))$. Lemma \ref{topology} implies that $f:Z\to X$ lifts to an embedding $\hat{f}:Z\hookrightarrow \hat{X}$. Since $Z$ is a compact space, by Scott's topological interpretation of separability (\cite{Sc}), there exists a finite cover $\bar{X}\to X$ such that $\hat{X}\to X$ factors through $\bar{X}$, and $f:Z\to X$ lifts to an embedding $\bar{f}:Z\hookrightarrow \bar{X}$.

We first prove the following lemma.

\begin{lemma}\label{cohomologyclass}
There exists a nontrivial cohomology class $\zeta\in H^1(\bar{X};\mathbb{Z})$, such that $\zeta|_{\bar{f}_*(\pi_1(Z))}=0$, as a homomorphism from $\pi_1(\bar{X})$ to $\mathbb{Z}$.
\end{lemma}

\begin{proof}

Each vertex piece of $\bar{X}$ is a finite volume hyperbolic $3$-manifold (finite cover of $N_1$ or $N_2$), and its intersection with $Z$ is a (possibly disconnected) oriented surface. So we get a cohomology class in each vertex piece, and we will show that these cohomology classes in vertex pieces can be "pasted together" to get the desired cohomology class on $\bar{X}$.

Each edge space of $\bar{X}$ is a finite cover of $\tilde{N}_1'\times [0,1]$ or $\tilde{N}_1''\times [0,1]$. We work with one edge space, and without loss of generality, we assume that this edge space is a finite cover of $\tilde{N}_1'\times [0,1]$, and denote it by $\bar{N}_1'\times [0,1]$. This edge space connects two vertex spaces of $\bar{X}$, and we denote them by $\bar{N}_1$ and $\bar{N}_2$. Then they are finite covers of $N_1$ and $N_2$ respectively.

Since $Z$ intersects with all vertex pieces of $\hat{X}$, $Z$ also intersects with all vertex pieces of $\bar{X}$. So $Z\cap \bar{N}_1=S_{1,1}$ or $S_{1,2}$, and we assume that $Z\cap \bar{N}_1=S_{1,1}$. Then $\bar{N}_1\to N_1$ is a $d$-sheet cyclic cover along $S_{1,1}$ for some $d\in \mathbb{Z}^+$. It is easy to check that, in the dual graph of $\bar{X}$, there are $(d,n+1)$ edges with marking $1$ that go through the vertex corresponding to $\bar{N_1}$. Here $(d,n+1)$ denotes the greatest common divisor of $d$ and $n+1$. Moreover, condition (7) of Proposition \ref{constructnonseparable} implies that each edge space as above contains exactly $\frac{n+1}{(d,n+1)}$ many $\bar{f}$-images of intervals in $Z$. The proof of this claim is same with the argument in Proposition 4.10 of \cite{Sun}. It is an elementary application of covering space theory, so we do not give the proof here.

So there are exactly $\frac{n+1}{(d,n+1)}$ many intervals in $Z$ that are mapped to $\bar{N}_1'\times [0,1]$, and it implies that $\bar{N}_2$ contains exactly $\frac{n+1}{(d,n+1)}$ many $\bar{f}$-images of $S_{2,j}$ in $Z$. Moreover, $\bar{N}_1\to N_1$ is a $d$-sheet cyclic cover implies $\bar{N}_1'\to N_1'$ is a $\frac{d}{(d,n+1)}$-sheet cyclic cover on the edge space. Since $\gamma_2|_{A'}:A'\to \mathbb{Z}$ is surjective, $\bar{N}_2\to N_2$ is a $\frac{d}{(d,n+1)}$-sheet cyclic cover corresponding to $\gamma_2\in H^1(N_2;\mathbb{Z})$.

Let $\bar{p}_1:\bar{N}_1\to N_1$ and $\bar{p}_2:\bar{N}_2\to N_2$ be the restrictions of the covering map $\bar{X}\to X$. Then $S_{1,1}\subset \bar{N}_1$ is dual to the cohomology class $$\frac{1}{d}(\bar{p}_1)^*(\gamma_{1,1})\in H^1(\bar{N}_1;\mathbb{Z}),$$ and the union of $\frac{n+1}{(d,n+1)}$ many $S_{2,j}$ in $\bar{N}_2$ is dual to the cohomology class $$\frac{n+1}{(d,n+1)}\cdot\frac{1}{\frac{d}{(d,n+1)}}(\bar{p}_2)^*(\gamma_2)=\frac{n+1}{d}(\bar{p}_2)^*(\gamma_2)\in H^1(\bar{N}_2;\mathbb{Z}).$$ Since $\gamma_{1,1}|_{A'}=(n+1)\gamma_2|_{A'}$, the restriction of these two cohomology classes give the same homomorphism from $\pi_1(\bar{N}'_1)<A'$ to $\mathbb{Z}$. So the cohomology classes on $\bar{N}_1$ and $\bar{N}_2$ agree with each other on $\bar{N}'_1\times [0,1]$.

We have proved that the cohomology classes on vertex spaces of $\bar{X}$ defined by oriented surfaces in $Z$ agree with each other on edge spaces. The Mayer-Vietoris sequence implies these cohomology classes on vertex space can be pasted to a cohomology class on $\bar{X}$. However, such a class is not unique, and may not vanish on $\bar{f}_*(\pi_1(Z))$.

Now we construct a homotopic nontrivial map $\pi:\bar{X}\to S^1$ such that the composition $\pi\circ \bar{f}:Z\to \bar{X}\to S^1$ is a constant map. Then we get the desired nontrivial cohomology class $\zeta\in H^1(\bar{X};\mathbb{Z})$ by pulling-back a generator of $H^1(S^1;\mathbb{Z})$ via $\pi^*$, since $\zeta|_{\bar{f}_*(\pi_1(Z))}=0$ holds.

On each vertex space $\bar{X}_v$ of $\bar{X}$, since $Z\cap \bar{X}_v$ is a fibered surface of $\bar{X}_v$, it induces a map $\pi|_{\bar{X}_v}:\bar{X}_v\to S^1$ such that $\pi|_{\bar{X}_v}(Z\cap \bar{X}_v)=1=e^{i0}\in S^1$. Now we take an edge space $\bar{X}_e$ of $\bar{X}$, and assume $\bar{X}_e=\bar{N}_1'\times [0,1]$. Since we have proved that induced cohomology classes from two adjacent vertex spaces of $\bar{N}_1'\times [0,1]$ agree with each other, the induced maps $\bar{N}_1'\times \{0\}\to S^1$ and $\bar{N}_1'\times \{1\}\to S^1$ are homotopy to each other.

So there is a map $\pi'':\bar{N}_1'\times [0,1]\to S^1$ extending the maps already defined on $\bar{N}_1'\times \{0,1\}$ induced from vertex spaces. The problem is that $\pi''$ may not map edges of $Z$ in $\bar{N}_1'\times [0,1]$ to $1\in S^1$, even up to homotopy. Take two points $p_0\in c'\times \{0\}$ and $p_1\in c'\times \{1\}$, such that there is an interval $I\subset Z$ such that $\bar{f}(0)=p_0$ and $\bar{f}(1)=p_1$ for end points $0,1\in I$. Then $\sigma=\pi''\circ \bar{f}|_{I}:I\to S^1$ maps both $0,1\in I$ to $1\in S^1$. Let $\sigma':I\to S^1$ be the inverse path of $\sigma$ in $S^1$, then we define a map $\pi':\bar{N}_1'\times [0,2]\to S^1$ by
$$\pi'(n,t)=\left\{
\begin{array}{cc}
\pi''(n,t) & \text{if\ } t\in [0,1]\\
\sigma'(t-1)\cdot\pi''(n,1) & \text{if\ } t\in [1,2].
\end{array}
\right.
$$ Here the operation in $\sigma'(t-1)\cdot\pi''(n,1)$ is the multiplication of $S^1$. The map $\pi'$ has the same definition as $\pi''$ on the two boundaries of $\bar{N}_1'\times [0,2]$, and its restriction on $f(I)\cup (\{p_1\}\times [1,2])$ is a null-homotopic map to $S^1$ relative to the boundary.

For any other interval $I'\subset Z$ with $\bar{f}(I')\subset \bar{N}_1'\times [0,1]$, let $\bar{f}(I')\cap (\bar{N}_1'\times \{1\})=\{p_1'\}$. Condition (8) of Proposition \ref{constructnonseparable} implies pairings preserve cyclic orders and intervals in $c'\times [0,1]$ are disjoint from each other, so the null-homotopy condition is also true for $\bar{f}(I')\cup (\{p_1'\}\times [1,2])\subset \bar{N}_1'\times [0,2]$. Then we can homotopy $\pi'$ relative to $\bar{N}_1'\times \{0,2\}$ and resize $\bar{N}_1'\times [0,2]$ to $\bar{N}_1'\times [0,1]$ to get the desired $\pi|_{\bar{X}_e}:\bar{X}_e=\bar{N}_1'\times [0,1]\to S^1$, such that it maps $\bar{f}(Z)\cap \bar{X}_e$ to $1\in S^1$.

So we get a map $\pi:\bar{X}\to S^1$ such that $\pi\circ \bar{f}:Z\to S^1$ is a constant map. Since the restriction of $\pi$ on each vertex space corresponds to a nontrivial first cohomology class, $\pi$ is homotopically nontrivial. So the proof is done.

\end{proof}

Now we continue with the proof of Proposition \ref{proofnonseparable}.

Let $k_1,\cdots,k_s$ be positive integers such that the restriction of $\zeta$ on each vertex or edge space of $\bar{X}$ is a $k_i$-multiple of a primitive cohomology class. Let $K$ be the least common multiple of $k_1,\cdots,k_s$. Then we take the cyclic cover of $\bar{X}$ corresponding to the kernel of $\pi_1(\bar{X})\xrightarrow{\zeta}\mathbb{Z}\to \mathbb{Z}/K\mathbb{Z}$, and get a finite cover $q:\breve{X}\to \bar{X}$ such that $Z$ lifts to be embedded into $\breve{X}$. Then $\frac{1}{K}q^*(\zeta)$ is a primitive cohomology class in $H^1(\breve{X};\mathbb{Z})$, and its restriction on each edge and vertex space of $\breve{X}$ is also primitive. Let $p:\breve{X}\to X$ be the induced finite cover from $\breve{X}$ to $X$. 

Let $\breve{N}_1$ and $\breve{N}_2$ be the elevations of $N_1$ and $N_2$ in $\breve{X}$ such that $S_{1,1}$ is contained in $\breve{N}_1$ and $S_{2,1}$ is contained in $\breve{N}_2$. Since $I_1$ and $I_2$ are edges in $Z$ connecting $S_{1,1}$ and $S_{2,1}$ with $f(I_1)\subset \tilde{N}_1'\times [0,1]$ and $f(I_2)\subset \tilde{N}_1''\times [0,1]$, there are edge spaces $\breve{N}_1'\times [0,1]$ and $\breve{N}_1''\times [0,1]$ in $\breve{X}$ connecting $\breve{N}_1$ and $\breve{N}_2$. Then $\breve{N}_1'$ and $\breve{N}_1''$ are finite covers of $\tilde{N}_1'$ and $\tilde{N}_1''$ respectively.

Since the restriction of $\zeta$ on $\bar{N}_1\cup \bar{N}_1'\times [0,1]\cup \bar{N}_1''\times [0,1]\subset \bar{X}$ is dual to the fibered surface $S_{1,1}\subset \bar{N}_{1,1}$, and $S_{1,1}$ is also a fibered surface in $N_1$, the restricted map $$p|:\breve{N}_1\cup \breve{N}_1'\times [0,1]\cup \breve{N}_1''\times [0,1]\to N_1\cup \tilde{N}_1'\times [0,1]\cup \tilde{N}_1''\times [0,1]$$ is a finite cyclic cover on each piece. The finite cyclic covers on these three pieces correspond to the cohomology class $\gamma_{1,1}\in H^1(N_1;\mathbb{Z})$ and its restrictions on $A'$ and $A''$ respectively.

Let $p|_{\breve{N}_1}:\breve{N}_1\to N_1$ be the restriction of $p$ on $\breve{N}_1$, and suppose this is a degree-$D$ cover. Then the surface $S_{1,1}$ in $\breve{N}_1$ is dual to $\frac{1}{D}(p|_{\breve{N}_1})^*(\gamma_{1,1})$, which is also equal to $\frac{1}{K}q^*(\zeta)|_{\pi_1(\breve{N}_1)}$.
Since $\gamma_{1,1}|_{A'}$ is an $(n+1)$-multiple of a primitive cohomology class in $H^1(N_1';\mathbb{Z})$, $\gamma_{1,1}|_{A''}$ is an $n$-multiple of a primitive class in $H^1(N_1'';\mathbb{Z})$, and the restriction of $\frac{1}{K}q^*(\zeta)|_{\pi_1(\breve{N}_1)}=\frac{1}{D}(p|_{\breve{N}_1})^*(\gamma_{1,1})$ on $\pi_1(\breve{N}_1')$ and $\pi_1(\breve{N}_1'')$ are primitive classes, we have that $D$ is a multiple of $n(n+1)$ and the following equation hold:
$$(n+1)\cdot \text{deg}(\breve{N}_1'\to N_1')=\text{deg}(\breve{N}_1\to N_1)=n\cdot\text{deg}(\breve{N}_1''\to N_1'').$$

Similarly, by applying the same argument to $\breve{N}_2\to N_2$, we get $$\text{deg}(\breve{N}_1'\to N_1')=\text{deg}(\breve{N}_2\to N_2)=\text{deg}(\breve{N}_1''\to N_1'').$$

So we get a contradiction, and $f_*(\pi_1(Z))<\pi_1(X)$ is not separable.
\end{proof}

\subsection{Proof of Theorem \ref{main}}

In this subsection, we give the proof of Thoerem \ref{main}. The proof is just a combination of Proposition \ref{matchingfibering}, Proposition \ref{2V2E}, Proposition \ref{constructnonseparable} and Proposition \ref{proofnonseparable}.

\begin{proof}
We start with a nontrivial geometrically finite amalgamation $\pi_1(M_1)*_A\pi_1(M_2)$ of finite volume hyperbolic $3$-manifolds $M_1$ and $M_2$. Then Proposition \ref{matchingfibering} produces a subgroup of $\pi_1(M_1)*_A\pi_1(M_2)$ with a two-vertex one-edge dual graph and an algebraically fibered structure. For simplicity, we still denote it by $\pi_1(M_1)*_A\pi_1(M_2)$.

Then Proposition \ref{2V2E} gives us a further subgroup $\pi_1(N_1)*_{A',A''}\pi_1(N_2)$ of $\pi_1(M_1)*_A\pi_1(M_2)$, which has a two-vertex two-edge graph of group structure, and an induced algebraically fibered structure. The subgroup $\pi_1(N_1)*_{A',A''}\pi_1(N_2)$ also satisfies other conditions in Proposition \ref{2V2E}.

In subsection \ref{space}, we constructed a space $X$ such that $\pi_1(X)\cong \pi_1(N_1)*_{A',A''}\pi_1(N_2)$. In Proposition \ref{constructnonseparable}, we constructed a generalized immersed surface $(Z,f)$ in $X$. Then Proposition \ref{proofnonseparable} implies that $f_*(\pi_1(Z))$ is a nonseparable subgroup of $\pi_1(X)\cong \pi_1(N_1)*_{A',A''}\pi_1(N_2)$.

Lemma \ref{subgroup} implies that $f_*(\pi_1(Z))$ is also a nonseparable subgroup of $\pi_1(M_1)*_A\pi_1(M_2)$, thus $\pi_1(M_1)*_A\pi_1(M_2)$ is not LERF.
\end{proof}

\section{NonLERFness of closed arithmetic hyperbolic $4$-manifold groups}\label{arithmetic4section}

In this section, we prove that Theorem \ref{main} implies Theorem \ref{arithmetichyperbolic4}. Actually, Theorem \ref{main} also directly implies Corollary \ref{arithmetichyperbolic}, by applying a similar argument.

The proof of Theorem \ref{arithmetichyperbolic4} follows the same idea in \cite{Sun}, except that we have a stronger nonLERFness result (Theorem \ref{main}) in this paper.

\begin{proof}
Let $M$ be a closed arithmetic hyperbolic $4$-manifold, then \cite{VS} implies that $M$ is an arithmetic hyperbolic manifold of simplest type. So there exist a totally real number field $K$ and a nondegenerate quadratic form $f:K^{5}\to K$ defined over $K$, such that the negative inertial index of $f$ is $1$, and $f^{\sigma}$ is positive definite for any non-identity embedding $\sigma:K\to \mathbb{R}$. Moreover, $\pi_1(M)$ is commensurable with $SO_0(f;\mathcal{O}_K)$. So to prove $\pi_1(M)$ is not LERF, we need only to prove $SO_0(f;\mathcal{O}_K)$ is not LERF.

We diagonalize the quadratic form $f$ such that the symmetric matrix defining $f$ is $A=\text{diag}(k_1,k_2,k_3,k_4,k_5)$ with $k_1,k_2,k_3,k_4>0$ and $k_5<0$.

The quadratic form $f$ has two quadratic subforms defined by $\text{diag}(k_1,k_2,k_3,k_5)$ and $\text{diag}(k_1,k_2,k_4,k_5)$ respectively. These two subforms satisfy the conditions for defining arithmetic groups in $\text{Isom}_+(\mathbb{H}^3)$, and we denote them by $f_1$ and $f_2$ respectively.

$SO_0(f_1;\mathcal{O}_K)$ and $SO_0(f_2;\mathcal{O}_K)$ are both subgroups of $SO_0(f;\mathcal{O}_K)<\text{Isom}_+(\mathbb{H}^4)$. Each of them fix a $3$-dimensional totally geodesic hyperplane in $\mathbb{H}^4$, and we denote them by $P_1$ and $P_2$ respectively. Then $P_1$ and $P_2$ intersect with each other perpendicularly along a $2$-dimensional totally geodesic plane $P$. So $M_i=P_i/SO_0(f_i;\mathcal{O}_K)$ is a hyperbolic $3$-orbifold for each $i=1,2$. Moreover, it is easy to see that $SO_0(f_1;\mathcal{O}_K)\cap SO_0(f_2;\mathcal{O}_K)=SO_0(f_3;\mathcal{O}_K)$, with $f_3$ be defined by $\text{diag}(k_1,k_2,k_5)$. Then $SO_0(f_3;\mathcal{O}_K)$ is the subgroup of $SO_0(f;\mathcal{O}_K)$ that fixes $P$. It is easy to see that $P/SO_0(f_3;\mathcal{O}_K)$ is a totally geodesic $2$-orbifold $\Sigma$ in $M_1\cap M_2$.

We first take a torsion-free finite index subgroup $\Lambda<\pi_1(\Sigma)$, and consider it as subgroups of $\pi_1(M_1)$ and $\pi_1(M_2)$. By applying LERFness of hyperbolic $3$-manifold (orbifold) groups, we get torsion free finite index subgroups $\Lambda_i<SO(f_i;\mathcal{O}_K)$ for $i=1,2$, with $\Lambda_1\cap SO_0(f_3;\mathcal{O}_K)=\Lambda_2\cap SO_0(f_3;\mathcal{O}_K)=\Lambda$, and $S=P/\Lambda$ has large enough product neighborhood in $N_1=P_1/\Lambda_1$ and $N_2=P_2/\Lambda_2$. Then there is an obvious map from $N_1\cup_S N_2$ to $\mathbb{H}^4/SO_0(f;\mathcal{O}_K)$, and the induced map on the fundamental group $\Lambda_1*_{\Lambda}\Lambda_2\to SO_0(f;\mathcal{O}_K)$ is injective

So $SO_0(f;\mathcal{O}_K)$ contains a subgroup isomorphic to $\Lambda_1*_{\Lambda}\Lambda_2$. $\Lambda_1*_{\Lambda}\Lambda_2$ is the fundamental group of $N_1\cup_S N_2$, which is the quotient space of two closed hyperbolic $3$-manifolds $N_1, N_2$ by pasting along a closed totally geodesic subsurface $S$. Any totally geodesic subsurface in a hyperbolic $3$-manifold gives a geometrically finite subgroup in the $3$-manifold group. By Theorem \ref{main}, $\Lambda_1*_{\Lambda}\Lambda_2$ is not LERF, so Lemma \ref{subgroup} implies $SO_0(f;\mathcal{O}_K)$ and $\pi_1(M)$ are not LERF.
\end{proof}

As in \cite{Sun}, we have the following corollary of Corollary \ref{arithmetichyperbolic}, which is about nonLERFness of non-arithmetic hyperbolic manifold groups. This corollary covers all known examples of non-arithmetic hyperbolic manifolds with dimension $\geq 5$.

\begin{corollary}
For the following non-arithmetic hyperbolic manifolds:
\begin{itemize}
  \item $M^m$ is a nonarithmetic hyperbolic $m$-manifold given by constructions in \cite{GPS} or \cite{BT}, with $m\geq 5$,
  \item $M^m$ is a nonarithmetic hyperbolic $m$-manifold given by the reflection group of some finite volume polyhedron in $\mathbb{H}^m$, with $m\geq 4$,
\end{itemize}
the fundamental group $\pi_1(M)$ is not LERF.
\end{corollary}

\section{Further questions}\label{furtherquestions}

In the proof of Theorem \ref{main}, the nonseparable subgroups we constructed are usually infinitely presented. Moreover, the nonseparable subgroups we constructed for closed arithmetic hyperbolic $4$-manifold groups are always infinitely presented. So we ask the following two questions on the existence of finitely presented nonseparable subgroups.

\begin{question}
For a general group $A$, do geometrically finite amalgamations $\pi_1(M_1)*_A \pi_1(M_2)$ of finite volume hyperbolic $3$-manifold groups contain finitely presented nonseparable subgroups?
\end{question}

\begin{question}
Do closed arithmetic hyperbolic $4$-manifold groups contain finitely presented nonseparable subgroups?
\end{question}

\bigskip

Theorem \ref{main} implies that nontrivial geometrically finite amalgamations of finite volume hyperbolic $3$-manifold groups are not LERF. Then it is natural to ask about amalgamations of finite volume hyperbolic $3$-manifold groups along a group that is geometrically infinite in at least one of the vertex groups.

\begin{question}
Let $M_1,M_2$ be two finite volume hyperbolic $3$-manifolds, $A$ be a nontrivial group, $i_1:A\to \pi_1(M_1)$ and $i_2:A\to \pi_1(M_2)$ be two injective group homomorphisms. If $i_2(A)<\pi_1(M_2)$ is a geometrically infinite subgroup, then in what circumstance, is $\pi_1(M_1)*_{A}\pi_1(M_2)$ LERF?
\end{question}

Since all geometrically infinite subgroups of a finite volume hyperbolic $3$-manifold group are virtual fibered subgroups, $A$ is always a surface group or a free group. So in this case, the abstract group structure of the edge group is not very complicated.

Actually, all results in Section \ref{nonseparablesubgroup} still work if $i_1(A)<\pi_1(M_1)$ is geometrically finite and $i_2(A)<\pi_1(M_2)$ is geometrically infinite (we need to modify the proof of Proposition \ref{2V2E} a little bit). So there are examples of nonLERF amalgamations of two finite volume hyperbolic $3$-manifold groups along a subgroup that is geometrically finite in one group and geometrically infinite in the other group. However, in this case, it seems not easy to get the algebraically fibered structure as in Section \ref{algebraicfibering}, so the general case is difficult to deal with. For a geometrically infinite amalgamation, condition (5) in Proposition \ref{2V2E} never holds, so our current technique is not applicable in this case.

\bigskip

In \cite{Sun}, for $3$-manifolds with empty or tori boundary, a topological criterion on LERFness of groups of such $3$-manifold is proved (in terms of geometric structures on $3$-manifolds). Results of \cite{Sun} and this paper together imply that, for almost all arithmetic hyperbolic manifolds with dimension $\geq 4$ (with possible exceptions in arithmetic hyperbolic $7$-manifolds defined by the octonion), their fundamental groups are not LERF. So we ask the following questions on more general $3$-manifolds and more general finite volume high dimensional hyperbolic manifolds.

\begin{question}\label{general3manifold}
For compact $3$-manifolds with higher genus boundary, is there a topological criterion on LERFness of their fundamental groups?
\end{question}

\begin{question}\label{generalhyperbolic}
Whether all finite volume hyperbolic manifolds with dimension at least $4$ have nonLERF fundamental groups?
\end{question}

For Question \ref{general3manifold}, the author propose the following possible criterion. For a compact $3$-manifold with higher genus boundary, we first do sphere and disc decompositions, then do the torus decomposition without doing the annulus decomposition. Under the torus decomposition, if there are two adjacent pieces that are Seifert fibered spaces or finite volume hyperbolic $3$-manifolds (geometric $3$-manifold with tori boundary), then the result in \cite{Sun} implies this manifold has nonLERF fundamental group. If the above picture does not show up, then the author expects the fundamental group to be LERF.

Daniel Groves informed the author that, under the torus decomposition of $M$, if all pieces are hyperbolic $3$-manifolds with higher genus boundary, then $\pi_1(M)$ is LERF. This follows from \cite{BW} and \cite{PW}. In \cite{BW}, it is proved that, for a $3$-manifold $M$ as above, all finitely generated subgroups of $\pi_1(M)$ are quasiconvex, and \cite{PW} shows that such quasiconvex subgroups are separable.

For Question \ref{generalhyperbolic}, the main difficulty is that we do not know whether the fundamental group of a general finite volume hyperbolic manifold with dimension $\geq 4$ contains $3$-manifold groups as its subgroups. The author also has no idea whether this is true for $7$-dimensional arithmetic hyperbolic manifolds defined by the octonion. It this is true, then probably a similar argument as the proof of Theorem \ref{arithmetichyperbolic4} would work for a general finite volume hyperbolic manifold with dimension $\geq 4$. Theorem \ref{main} may also be useful for proving some (arithmetic) lattices in some other (semisimple) Lie groups are not LERF.

\bigskip

In Proposition \ref{fiberedboundaryslope}, we proved that, for any boundary component $T$ of a cusped hyperbolic $3$-manifold, all but finitely many boundary slopes on $T$ are virtually fibered boundary slopes. So it is natural to ask the following question.

\begin{question}
For a cusped hyperbolic $3$-manifold $M$ and a boundary component $T\subset \partial M$, are all slopes on $T$ virtually fibered boundary slopes?
\end{question}

Proposition \ref{arithmeticslope} implies the answer for cusped arithmetic hyperbolic $3$-manifolds is positive. However, it seems not easy to prove it for a general cusped hyperbolic $3$-manifold. 

If we want to use the proof of Proposition \ref{fiberedboundaryslope}, then we need to know more about the shape of the Thurston norm unit ball when taking a finite cover. If in some finite cover $M'$ of $M$, all codimension-$1$ hyperplanes in $H^1(M';\mathbb{R})$ intersects with some open face of the Thurston norm unit ball, then all slopes on $T$ are virtually fibered boundary slopes. However, it seems we have very few knowledge on the shape of Thurston norm unit ball when taking a finite cover such that the first betti number increases.

If one want to develop a version of Agol's construction of virtual fibered structure in \cite{Ag2} relative to a fixed slope on $T$, then there is some difficulty on the behavior of norm-minimizing surfaces when doing Dehn filling along the fixed slope.

\end{document}